\documentclass[a4paper]{amsart}
\usepackage[utf8]{inputenc}
\usepackage[english]{babel}
\usepackage[sort]{cite}

\usepackage{amsmath,amssymb,amsfonts}
\usepackage{graphicx}
\usepackage{xcolor}

\usepackage{caption}
\usepackage{color}
\usepackage{multirow}
\usepackage{url}
\usepackage{booktabs}

\usepackage{tikz}
\usetikzlibrary{positioning}

\usepackage{MnSymbol}

\usepackage[normalem]{ulem}

\usepackage{cleveref}
\crefname{equation}{}{}
\crefname{equation}{}{}
\crefname{figure}{Figure}{Figure}
\crefname{section}{Section}{Section}
\crefname{proposition}{Proposition}{Proposition}
\crefname{theorem}{Theorem}{Theorem}
\crefname{corollary}{Corollary}{Corollaries}
\crefname{definition}{Definition}{Definition}
\crefname{lemma}{Lemma}{Lemma}
\crefname{notation}{Notations}{Notation}
\crefname{remark}{Remark}{Remark}
\crefname{claim}{Claim}{Claim}
\crefname{assumption}{Assumption}{Assumption}

\newtheorem{theorem}{Theorem}[section]
\newtheorem{corollary}[theorem]{Corollary}

\newtheorem{definition}[theorem]{Definition}

\newtheorem{proposition}[theorem]{Proposition}

\newtheorem{remark}[theorem]{Remark}
\newtheorem{lemma}[theorem]{Lemma}

\newcommand{\N}{\mathbb{N}}
\newcommand{\R}{\mathbb{R}}
\newcommand{\F}{\mathbb{F}}

\newcommand{\Z}{\mathbb{Z}}

\newcommand{\emin}{\ \mathrm{e_{min}}}
\newcommand{\emax}{\ \mathrm{e_{max}}}
\newcommand{\fl}{\mathrm{fl}_u}
\newcommand{\tmu}{\tilde{\kappa}}
\newcommand{\vect}[1]{\mathbf{#1}}
\newcommand{\dist}{\operatorname{dist}}
\newcommand{\diag}{\operatorname{diag}}
\newcommand{\deriv}[2]{\mathrm{D}_{#1} {#2}}

\usepackage{marginnote}

\title{When can forward stable algorithms be composed stably?}

\author{Carlos Beltr\'an}
\author{
Vanni Noferini}
\author{
Nick Vannieuwenhoven}

\thanks{Universidad de Cantabria, Avda. Los Castros s/n, Santander, Spain, \texttt{beltranc@unican.es}. Supported by Grant PID2020-113887GB-I00 funded by MCIN/ AEI /10.13039/501100011033, and by the Banco Santander and Universidad de Cantabria grant 21.SI01.64658.}
\thanks{Aalto University, Department of Mathematics and Systems Analysis, P.O. Box 11100, FI-00076 Aalto,
Finland, \texttt{vanni.noferini@aalto.fi}. Supported by an Academy of Finland grant (Suomen Akatemian p\"a\"at\"os 331240).}
\thanks{KU Leuven, Department of Computer Science, Celestijnenlaan 200A, B-3001 Leuven, Belgium, \texttt{nick.vannieuwenhoven@kuleuven.be}. Supported by a Postdoctoral Fellowship of the Research Foundation---Flanders (FWO) with project 12E8119N. Part of this work was initiated while visiting the Universidad de Cantabria supported by the FWO Grant for a long stay abroad V401518N.}

\begin{document}

\begin{abstract}
We state some widely satisfied hypotheses, depending only on two functions $g$ and $h$, under which the composition of a stable algorithm for $g$ and a stable algorithm for $h$  is a stable algorithm for the composition $g \circ h$.
\end{abstract}
\maketitle


\section{Introduction}\label{sec:intro}

As computers became increasingly common tools for engineers, mathematicians, and scientists, the \emph{accuracy} of numerical computations in floating-point arithmetic has been of great interest. It was realized early in the history of numerical analysis that mathematically equivalent algorithms can nevertheless behave in widely different ways when implemented in floating-point arithmetic. To understand this behavior, the concept of \emph{stability of an algorithm} was introduced, which categorizes how much an algorithm can be trusted when executed in finite-precision arithmetic.
{Popular} references such as \cite{Higham1996, Trefethenbook} consider three different types that can be summarized at a high level as follows:
\begin{enumerate}
\item A \emph{forward stable} algorithm (also called \emph{weakly stable} in \cite{Bunch1987}) gives an answer that is close to the actual exact answer for the given input.

\item A \emph{mixed forward-backward stable} algorithm (sometimes called \emph{numerically stable} \cite{Jong77,Higham1996}) gives an almost exact answer to a nearby problem. 

\item A \emph{backward stable} algorithm (also called \emph{strongly stable} in \cite{Bunch1987}) provides the exact answer to a nearby problem. 
\end{enumerate}
{Rigorous} definitions of these types of stability are presented in \cref{sec_numerical_stability}.

A {classical} example of mathematically equivalent algorithms with different backward stability properties is found in solving the overdetermined least-squares problem $A\mathbf{x}=\mathbf{b}$, where $A \in \mathbb{R}^{m \times n}$ is a left-invertible rectangular matrix.
A backward stable algorithm is obtained by computing a reduced singular value decomposition of $A=USV^T$, solving the diagonal system $S\mathbf{y}=U^T\mathbf{b}$, and computing $\mathbf{x} = V \mathbf{y}$ \cite{Bjorck2015}. Note that the individual steps in this composition are themselves  stable algorithms. On the other hand, transforming the overdetermined system into the equivalent linear system of normal equations $(A^T A)\mathbf{x} = A^T\mathbf{b}$ and solving it by Gauss elimination with full pivoting does not result in a backward stable algorithm when the matrix condition number $\kappa(A)=\|A\|_2 \|A^{\dagger}\|_2$ is large \cite{Bjorck2015}; herein, $A^\dagger$ denotes the pseudoinverse and $\|A\|_2$ the spectral norm of $A$. Interestingly, both of these algorithms are \emph{also} individually stable, but in this case their composition yields a disappointingly unstable algorithm (for some inputs). Here, the culprit is known to be the condition number of the Gram matrix $A^T A$ which is the square of $\kappa(A)$ \cite{Higham1996}.

The above example suggests that an unstable algorithm resulted because of a much higher condition number in an intermediate step. However, this is not the only reason why a composition of stable algorithms can fail to be stable. For example,
\begin{align*}
g : \R \to \R,\, x \mapsto x / 3  \quad\text{and}\quad  h : \R \to [1, \infty),\, x \mapsto x^2 + 1
\end{align*}
are individually backward stable algorithms when implemented in the straightforward way. However, in IEEE Standard 754 double-precision floating-point arithmetic, computing $h\circ g$ at $0$ results in the floating-point representation $\mathrm{fl}(1/3)$ of $\frac{1}{3}$. Since $\mathrm{fl}(1/3) < \frac{1}{3}$, there is no real input $x'$ such that $(g \circ h
)(x') = \mathrm{fl}(1/3)$. Hence, $g \circ h$ is not a backward stable algorithm to compute $\frac{1}{3}(x^2 + 1)$.

The foregoing observations raise pertinent questions: \emph{When does composing stable algorithms yield another stable algorithm? Why does it work only in some cases and not in others? Are the above problems the only obstructions?} It appears these questions received little attention in the literature, even though a way to compose algorithms stably seems to us a cornerstone of any mathematical theory of stability. For backward stability, the question was investigated to some extent by Bornemann \cite{Bornemann}, whose approach is discussed in \cref{sec_related_work}.


In this paper, we identify two mild sufficient conditions based on condition numbers \cite{Rice1966,BC2013}, called \emph{amenability} and \emph{compatibility}, and prove that these are sufficient to stably compose \emph{forward stable} algorithms. The formal definition of the condition number $\kappa(f, x)$ of a function $f : D \to E$ at an input $x \in D$ is recalled in \cref{def:mu}. Our main result can be summarized informally as follows.

\begin{corollary}[Informal version of \cref{th:composition}]\label{thm_informal}
Let $f = g \circ h$ be a composition of functions. The composition of forward stable algorithms for $g$ and $h$ results in a forward stable algorithm for $f$ if all of the following conditions hold:
\begin{itemize}
    \item $1+\kappa(f,x)$ is not much smaller than $(1+\kappa(g, h(x)))(1+\kappa(h,x))$;
    \item the growth of the condition number $x \mapsto \kappa(\phi,x)$ when $x$ varies is essentially bounded by $\kappa(\phi,x)$ itself
    for both $\phi=g,h$; and
    \item the limit of $\kappa(\phi,x)$ is $\infty$ as $x$ tends to either the boundary of the domain of $\phi$ or an input $x'$ with $\kappa(\phi, x')=\infty$, for both $\phi=g,h$.
\end{itemize}
The first condition is called \emph{compatibility} and the last two conditions {are collectively called} \emph{amenability} of $g$ and $h$.
\end{corollary}

Compatibility and amenability only involve properties of the computational problems $f$, $g$, and $h$ themselves. This implies that if the problem $f$ is decomposed into amenable and compatible problems as $f = g \circ h$, then any choice of forward stable algorithms implementing $g$ and $h$ yields a corresponding forward stable algorithm for $f$.

\subsection{Why forward errors?}

Our approach proves forward stability of the composition of forward stable algorithms, but we could not find arguments yielding similar claims for the other standard notions of stability. Although forward stability is sometimes overlooked as too weak, we believe nevertheless that there are reasons to give it more credit:

\begin{enumerate}
\item When the goal is to prove instability of an algorithm, forward instability is actually the strongest concept, by contraposition.  {When} one of the conditions in \cref{thm_informal} fails, it is often possible to find specific inputs for which the algorithm is provably unstable.
\item Algorithms are increasingly being implemented in either a multiple or variable precision setting \cite{bf2000,br14,cahi18,fahi18,fahi19,high17-multiprec,hilu20}. Many of the variable-precision algorithms follow the philosophy of guaranteeing a {user-specified upper bound} on the forward error. A forward stable algorithm is useful for this goal even if it is not backward stable.
\item The final users of numerical algorithms are often not numerical analysts, but rather scientists and engineers. In our experience, forward stability and forward errors are easier to interpret for such non-experts than backward stability.
Forward stability additionally exposes the hardness of the problem as measured by the condition number, while a backward error analysis may inadvertently give a non-expert the impression that the numerically computed \emph{outputs} can be trusted regardless of the problem's condition.
\end{enumerate}

\subsection{Outline and contributions}

In the next section, we review Bornemann's approach to investigate the backward stability of composed algorithms using the problem's condition number.
A gap in the main result is identified, illustrating the importance of our amenability condition, at least for forward stability.

\cref{sec:cond} recalls Rice's condition number (\cref{def:mu}) and proves a simple but important condition number bound for a composition of functions $f=g\circ h$ in \cref{lem:mucompupper}. Then, Pryce's coordinatewise relative error metric is recalled in \cref{def:finsler}. This defines a \emph{metric} measuring relative error,
so that the corresponding condition number is one for relative error. The main result of this section is the characterization of the condition number in Pryce's metric in \cref{prop_cond_expression}, which reduces to the usual one for univariate real-valued functions.

The main result of this paper states that properties involving the above condition number can establish the stability of compositions of forward stable algorithms. For such a statement to be precise, however, we need to clarify what we mean by a ``numerical algorithm'' for a function $f$ and which exact definitions of forward, mixed, and backward stability we use. This is covered in \cref{sec:floatingpoint}.

\Cref{sec_main} is the heart of this paper, stating and proving the main result. First, we introduce the concept of \textit{amenable problems} in \cref{def_amenable}. This captures a wide class of problems that are well-behaved in their domain of definition and the growth of their condition number. Next, a \textit{compatibility condition} is introduced in \cref{def_compatible}. 
The main result, \cref{th:composition}, on the forward stability of a composition of forward-stable algorithms is proven under the assumption of amenable and compatible problems. Finally, we investigate the string of implications from backward to mixed to forward stability, under the hypothesis of amenability in \cref{thm_backward_implies_mixed,thm_mixed_implies_fp}.

All of the main results are employed in \cref{sec:amenandstable} to identify several elementary amenable problems and forward stable algorithms, in the coordinatewise relative error metric, to solve them. The surprising example of the sine function, a non-amenable function on $\R$ that cannot be realized as a composition of well-behaved (i.e., compatible and amenable) functions, is featured in \cref{sec_sinful}. Finally, our conclusions are presented in \cref{sec_conclusions}.


\section{Bornemann's approach} \label{sec_related_work}

The study of {compositions} of backward stable algorithms was initiated by Bornemann \cite{Bornemann}. His prime insight was that the stability of a composition $g \circ h$ of backward stable algorithms $g$ and $h$ can be gleaned from the condition number of the inverse function $h^{-1}$. In other words, he employed the problem-specific condition number for establishing algorithm-specific stability. Bornemann's results are based on the study of the relative error condition number
\[
\kappa_f(x)=\frac{|xf'(x)|}{|f(x)|}.
\]
He also introduces a \emph{stability indicator} to quantify how backward stable an algorithm is. If a numerical algorithm $f^u$ for a mathematical function $f$ {runs} in floating-point arithmetic with unit roundoff $u$, {Bornemann's stability} indicator is the smallest number $\beta_f$ such that
\[
f^u(x)=f(x+\Delta x),\quad \frac{|\Delta x|}{|x|}\leq \beta_f(x) u+O(u^2),
\]
where $\beta_f(x)$ can depend on the base point $x$. If there is no such $\Delta x$, $\beta_f(x)=\infty$ and the algorithm is not backward stable at $x$.

The main result from \cite{Bornemann} is the following connection between the stability indicator of a composition of functions $g \circ h$ and the condition number of $h^{-1}$ in case $h$ is invertible:
\begin{equation}\label{eq:betaf}
\beta_f( x ) \leq \beta_h(x) + \kappa_{h^{-1}}(h(x))\,\beta_g(h(x)).
\end{equation}
{While this approach helps in a number of examples, such as those} described in detail in \cite{Bornemann}, there are two main drawbacks:
\begin{itemize}
    \item Inequality \eqref{eq:betaf} needs $h$ to be invertible, while many functions we use in numerical analysis are not even defined between spaces of the same dimension.
    \item {Although not explicitly stated in \cite{Bornemann}, a careful analysis of the proof reveals that the result is only valid for sufficiently small values of $u$. Indeed, the coefficient $O(u^2)$ term in \eqref{eq:betaf} is not necessarily an absolute constant, but it may depend} on the dimension of the problem and even the base point $x$. Assume for example that an algorithm for computing $f(x)=x$ is designed with the property that
    $$
    f^u(x)=f(x+\Delta x),\quad \frac{|\Delta x|}{|x|}\sim u+2^{x^2+x^{-2}}u^2.
    $$
    The stability indicator is $1$ for every $x\in\R$ and according to this we could be tempted to say that ``for every $x$, the algorithm is perfectly backward stable at $x$''... but most if not all  numerical analysts would definitely say that the algorithm is not stable, since the constant hidden in the $O(u^2)$ term is huge and, even worse, unbounded. 
\end{itemize}

All in all, although Bornemann's {heuristics} is a first sensible approach to the problem of composing stable algorithms, a more formal approach is required to avoid {subtleties} as the ones described above

\section{The condition number in coordinatewise relative error}\label{sec:cond}

We restrict ourselves to computational problems that can be modeled by a function $f$ from a subset of one real vector space to another. We start by recalling Rice's definition of condition number \cite{Rice1966} in this context, assuming to have fixed (possibly different) distance functions on the domain and codomain of $f$.

Recall that given a subset $D\subseteq\R^d$, a map $\dist_D:D\times D\to [0,\infty]$ is a {\em distance function} if (i) it is symmetric, (ii) it satisfies the triangle inequality $\dist_D(x,z) \leq \dist_D(x,y) + \dist_D(y,z)$ for $x,y,z\in D$, and (iii) $\dist_D(x,y)=0 \Leftrightarrow x=y$. If there is no ambiguity on the nature of the set $D$, we omit the subscript and simply write $\dist$.

\begin{definition}[Condition number]\label{def:mu}
Given a function $f:S\to T$, where $S\subseteq\R^m$ and $T\subseteq\R^n$ and two distance functions $\dist_S$ and $\dist_T$ defined on $S$ and $T$ respectively, the condition number $\kappa(f,x)$ of $f$ at $x\in S$ is
\[
\kappa(f,x) = \lim_{\epsilon\to0} \sup_{\substack{y\in S,\\ 0<\dist_{S}(x,y)\leq \epsilon}}\frac{\dist_{T}(f(x),f(y))}{\dist_{S}(x,y)}\in[0,\infty].
\]
If $x$ is an isolated point for the topology induced by the metric on $S$, then we set by convention $\kappa(f,x)=0$.
We also define
\[
\tmu=1+\kappa.
\]
\end{definition}

\noindent The $\limsup$ in \cref{def:mu} is always defined since we accept the value $\infty$.

The condition number is a \textit{geometric invariant} that depends on the function $f$ but not on the algorithm we use to compute it. There is also an explicit dependence of the condition number on the choice of the distance functions $\dist_S$ and $\dist_T$.
In addition, condition numbers satisfy the following sensible composition inequality.

\begin{lemma}\label{lem:mucompupper}
 If $f = g\circ h$ is a composition of functions, then
	\begin{equation}\label{eq:boundmu}
	\tmu(f,x)\leq \tmu(g,h(x))\,\tmu(h,x).
	\end{equation}
\end{lemma}
\begin{proof}
	See \cref{sec:P0}
\end{proof}

The choice of the distances in \cref{def:mu} depends on the context.
As real arithmetic operations are substituted in a numerical algorithm by floating-point operations in the standard model of floating-point arithmetic \cite{Higham1996}, this results in small relative errors. Therefore, we focus on a metric that measures relative errors.

Let first us clarify what we mean by relative error. Usually, in numerical analysis, the relative error of an approximation $\tilde{x} \in \R$ to an exact value $x \in \R$ is defined as $\frac{|x-\tilde{x}|}{|x|}$.
Unfortunately, this does not define a distance on $\R$ because symmetry and the triangle inequality can fail. Therefore, neither \cref{def:mu} nor \cref{lem:mucompupper} should be expected to hold for this way of measuring relative error.
Instead, we adopt the \textit{coordinatewise relative error metric} that was already introduced by Pryce \cite{Pryce1984} in 1984, building on Olver's Rp error analysis \cite{Olver1978,Olver1982}. Herein, the distance in $\R^d$ between $x$ and $y$ is defined as the length of the shortest curve joining $x$ and $y$, where the length is ultimately determined by the choice of a \textit{Riemannian metric} on $\R$ that captures the concept of relative errors. This coordinatewise relative error metric and the surprising topology it induces on $\R^d$ is stated next.

\begin{definition}[Coordinatewise relative error metric \cite{Pryce1984}]\label{def:finsler}
The \emph{topology associated with the coordinatewise relative error in $\R^d$} is defined by $3^d$ connected components $U_j$ in $\R^d$, corresponding to the sign pattern (an element of $\{-1,0,1\}^d$) of each vector. The connected component $U_0=\{0\}$ consists of just one point. On each of the other components $U_1,\ldots,U_{3^d-1}$, we define a \emph{Riemannian metric} as follows. Let $v,w$ be a pair of tangent vectors lying in the subspace spanned by the elements of $U_j$ and based at the point $x\in U_j$. We associated with them the inner product
\[
\langle v,w\rangle_x=\sum_{i\in\chi(x)} \frac{v_iw_i}{|x_i|^2},
\]
where $\chi(x) = \{ i \mid x_i\neq0 \}$ contains the nonzero indices of $x$, which defines the tangent space to $U_j$ at $x$. The induced norm of a tangent vector $v$ is $\|v\|_x=\langle v,v\rangle_x^{1/2}$. The {\em length} of a $C^1$ curve $\gamma:[a,b]\to U_j$ is then defined by
\[
\mathrm{Length}(\gamma)=\int_a^b\|\gamma'(t)\|_{\gamma(t)}\, \mathrm{d} t
= \int_a^b \left(\sum_{i\in\chi(\gamma(t))} \frac{v_i^2}{|\gamma_i(t)|^2}\right)^{1/2}\,\mathrm{d}t,
\]
where $\gamma'(t)=(v_1,\dots,v_d)$ is the derivative of $\gamma$ at $t$, and $\gamma_i$ is the $i$th component function of $\gamma$.
The \emph{distance} between any pair of points $x,y\in \R^d$ is $0$ if $x=y$, $\infty$ if $x\in U_j$ and $y\in U_k$ with $j\neq k$ lie on different connected components, and
\begin{equation}\label{eq:distancedef}
\dist(x,y)=\inf_{\gamma}\mathrm{Length}(\gamma)
\end{equation}
otherwise.
The infimum is taken over all $C^1$ curves contained in the connected component $U_j$ with extremes $x$ and $y$. This distance function satisfies the axioms of distance, including the triangle inequality $\dist(x,z)\leq \dist(x,y)+\dist(y,z)$ for all $x,y,z\in \R^d$.
\end{definition}

From now on whenever we refer to $\R^d$ we assume that it is endowed with the topology and the metric structure of \cref{def:finsler}.

The infimum of \eqref{eq:distancedef} is actually a minimum; that is, there exists a curve with extremes $x,y$ such that its length is equal to $\dist(x,y)$. This curve is called a \emph{minimizing geodesic} between these two points. The existence follows from the Hopf--Rinow Theorem, see for example \cite[Theorem 1.10]{Gromov}. In $\R$, the relative error distance between two points $x$ and $y$ with the same sign is either $\dist(0,0)=0$ or
\[
\dist(x,y)
=\int_{0}^1 \|x-y\|_{tx+(1-t)y}\, \mathrm{d}t
=\int_{0}^1 \frac{|x-y|}{|tx+(1-t)y|}\, \mathrm{d}t
=\left|\log\frac{x}{y}\right|.
\]
If $x$ and $y$ do not have the same sign, then $\dist(x,y)=\infty$. The coordinatewise relative error metric in $\R^d$ is the product metric of $\R$ when the latter is endowed with the relative error metric. Hence, the squared distance between two points $v=(v_1,\ldots,v_d)$ and $w=(w_1,\ldots,w_d)$ is
\(
\dist(v,w)^2
= \dist(v_1,w_1)^2+\cdots+\dist(v_d,w_d)^2.
\)
In particular, for $u\in(0,\frac{1}{4})$ we get after some manipulations that
if $w_i=v_i(1+\delta_i)$ for some $\delta_i\in(-u,u)$, then $\dist(v,w)< 2\sqrt{d}\,u$.
The following fact will also be helpful.
\begin{lemma}\label{lem:geodesicas}
	Let $f:S\to\R^n$ be as in \cref{def:mu}. Let $x,y\in S$. If there is a minimizing geodesic $\gamma(t)$ joining $x$ and $y$ and such that $\kappa(f,z)\leq C$ {for all $z \in \gamma$}, then
	\[
	\dist(f(x),f(y))\leq C\dist(x,y).
	\]
\end{lemma}
\begin{proof}
See \cref{sec:P1}.
\end{proof}

With the metric structure clarified, we further characterize the condition number from \cref{def:mu}. Recall that for $A\in\R^{n\times m}$ with singular value decomposition $A=U\Sigma V^T$, the Moore--Penrose pseudoinverse of $A$ can be defined as $A^\dagger=V \Sigma^\dagger U^T \in\R^{m\times n}$ where $\Sigma^\dagger$ is obtained by transposing $\Sigma$ and changing the non-zero elements $\sigma_i$ of $\Sigma$ to $\sigma_i^{-1}$. The pseudoinverse of the zero matrix $0 \in \R^{n\times m}$ is $0 \in \R^{m\times n}$.

\begin{proposition}[Condition number in coordinatewise relative error]\label{prop_cond_expression}
Let $f : S \to \R^n$ be a differentiable map with $S\subset\R^m$ open, and let the connected components $U_i\subset \R^m$ be as in \cref{def:finsler}. Let $\kappa(f,x)$ be the condition number from \cref{def:mu} with respect to the coordinatewise relative error metric on both the domain and codomain. Then, the following holds:
\begin{enumerate}

\item[(i)] if $x\in S$ and there exists an open neighborhood $N$ of $x$ such that $f(N)$ is contained in one connected component of $\R^n$, then
\begin{align} \label{eqn_mu_diff}
 \kappa(f,x) = \| \diag({f(x)})^\dagger \deriv{x}{f} \diag(x) \|_2,
\end{align}
where $\|\cdot\|_2$ is the spectral $2$-norm, $\deriv{x}{f}$ is the derivative of $f$ at $x$,
and $\diag(z)$ is the diagonal matrix whose entries are the elements of a vector $z\in\R^d$. In particular, if $f$ takes an open neighborhood of $x$ to a constant, then $\kappa(f,x)=0$.

\item[(ii)] in all other cases $\kappa(f,x)=\infty$.
\end{enumerate}
\end{proposition}
\begin{proof}
See \cref{sec:Pm1}.
\end{proof}
\begin{remark}
Item (ii) in \cref{prop_cond_expression} corresponds to $x\neq0$, $f_i(x)=0$ for some $i$ and yet $f_i(y)\neq0$ for some $y$ arbitrarily close to $x$. In all other cases, recalling also that $0$ is an isolated point in our chosen topology, item (i) applies, and then the presence of the Moore--Penrose pseudoinverse in \eqref{eqn_mu_diff} implies that only the non-zero components of $x$ and $f(x)$ contribute to $\kappa(f,x)$.
\end{remark}

Two important special cases of equation \cref{eqn_mu_diff} arise for $m=1$ and $n=1$, simplifying to respectively
\[
\kappa(f,x) = |x| \sqrt{ \sum_{i\in\chi(f(x))} \left( \frac{f_i'(x)}{|f_i(x)|} \right)^2 } \quad\text{and}\quad
\kappa(f,x) = \frac{1}{|f(x)|} \sqrt{\sum_{i\in\chi(x)} \left( x_i \frac{\partial f}{\partial x_i} \right)^2}.
\]
For univariate scalar-valued functions, i.e., $m=n=1$, both reduce to the familiar expression of the relative condition number
\begin{align} \label{eqn_usual_rel_cond}
\kappa(f,x)
=\frac{|x|\cdot |f'(x)|}{|f(x)|}.
\end{align}

\begin{remark}
  It follows from \cref{prop_cond_expression} that, if $f_i(x)=0$ for some $i$, then either $\kappa(f,x)=\infty$ or $f_i(x)=0$ in a neighborhood of $x$. In particular, this clarifies the special values in the renowned formula \eqref{eqn_usual_rel_cond}:
  \begin{itemize}
    \item If $x=0$ then $\kappa(f,x)=0$, regardless of the value of $f'(x)$ and $f(x)$.
    \item If $x\neq 0$ and $f'(x)=0$ in an open neighborhood of $x$, then $\kappa(f,x)=0$.
    \item If $x\neq0$, $f(x)=0$ but $f'(x)$ is not constantly $0$ in an open neighborhood of $x$, then $\kappa(f,x)=\infty$ even if $f'(x)=0$
  \end{itemize}
\end{remark}

\section{Numerical stability of algorithms}\label{sec:floatingpoint}

This section briefly states the precise definitions of forward, mixed and backward stability of numerical algorithms for a computational problem that we use in this paper.
Informally, an algorithm is a sequence of instructions that can be programmed in any programming language, like C, Julia, Matlab, or Python. Several of these programming languages offer multiple floating-point types, such as half, single, and double precision. Until recently, mathematical software was usually implemented in one fixed precision. This paradigm is starting to change \cite{bf2000,br14,cahi18,fahi18,fahi19,high17-multiprec,hilu20,Everyone2020} as we currently witness an expanding use of both high-precision (for enhanced accuracy) and low-precision (for enhanced speed) arithmetic; {sometimes, different precisions may even be used for different steps in the same algorithm.}
Therefore, the definition of numerical algorithm and stability below will not depend on any fixed precision, but rather holds for all sufficiently large precisions simultaneously.

In the next section, we use these definitions to prove a composition theorem for forward stability in \cref{th:composition}.

\subsection{Numerical algorithm}\label{sec_floating_point_arithmetic}\label{sec:algorithm}
Our analysis assumes that the standard model of floating-point arithmetic \cite{Higham1996} is adopted. For completeness, it is described in \cref{appendix:bss}.
The \emph{unit roundoff} of the floating-point number system $\F_u\subseteq\R$ is denoted by $u$. We also write $u = 2^{-t}$, where $t>2$ is the \emph{precision}. To prove our main results, we need a relaxation of the standard floating-point system: we assume that that the smallest possible exponent in the system is $\emin=-\infty$ and the largest possible exponent is $\emax=\infty$.

The standard floating-point system defines a \textit{roundoff map} $\fl : \R \to \F_u$ that takes $x \in\R$ to the number in $\F_u$ that is closest to $x$ in the Euclidean distance. Ties can be broken by the usual schemes, such as rounding to odd, to even, or away from $0$. Our results do not depend on this choice. From now on we assume that maps $\fl$ exist for all $0<u<\frac{1}{4}$.


Since our aim is to prove a composition theorem for numerical algorithms in arbitrary precision, we should agree on what constitutes an ``an algorithm for a computational problem,'' mostly to exclude pathological cases.
The definition of algorithm we adopt in this paper is the \emph{BSS machine}, proposed by Blum, Shub, and Smale \cite{BSS1989} and further developed in \cite{BCSS1998}; for completeness it is recalled in \cref{appendix:bss}. 
Following \cite{cucker2015,MalajovichShub}, for the purpose of this paper, a \textit{numerical algorithm} is a BSS machine where so-called \emph{approximate computations} \cite{cucker2015,MalajovichShub} are obtained by executing it in the standard model of floating-point arithmetic; for completeness, the standard details are included in \cref{appendix:bss}.

\subsection{Scalable functions}

Recall that the BSS model takes into account that many algorithms, especially those in numerical linear algebra, are \textit{scalable}: they can be applied to problems of arbitrary dimension.
In practice, mathematical algorithms like computing matrix factorizations, e.g., QR, LU, and SVD, of $n\times n$ matrices are usually implemented  for all $n\geq1$, rather than having separate implementations for each $n \in \N$. We will refer to the corresponding family of functions, which have input and output of varying dimensions, as a \emph{scalable function}.

\begin{definition}[A scalable function]
	A scalable function is of the form
	\[
	f:\bigcupdot_k S_k\to \bigcupdot_k \R^{n_k},
	\]
	where $S_k\subseteq\R^{m_k}$ and $m_k,n_k\in\N$ are sequences, either both infinite or both finite and with the same length, and $\cupdot$ denotes the disjoint union of sets.

 When applied to $x\in S_k$ the function can be denoted by $f_k:S_k\to \R^{n_k}$, but for brevity we simply write $f(x)$. From now on we denote $M_k=\max(m_k,n_k)$, that is the maximum of the dimension of the input and the output for every $k$, since in practice an element of $S_k$ is represented by a vector in $\R^{m_k}$.
\end{definition}

\subsection{Backward, mixed, and forward stability}\label{sec_numerical_stability}

We denote a specific BSS machine that attempts to implement a function $f$ by $\hat{f}$. The output (if any) of the corresponding numerical algorithm with unit roundoff $u$ on input $x$ is denoted by $\hat{f}^u(x)$. This emphasizes that the algorithm tries to approximate a function $f$ and that the computations are carried out in $\F_u$.

Numerical stability was introduced to classify which numerical algorithms $\hat{f}^u$ can be claimed to implement a given function $f$ reliably.
In the context of scalable functions and numerical algorithms with arbitrary precision $t$, we believe that the magnitude of errors should be viewed relative to the size of the domain $\R^{m_k}$ and codomain $\R^{n_k}$.
Consequently, having defined $M_k=\max \{m_k,n_k\}$, the definitions of numerical stability below allow errors to grow with $M_k$, but at most polynomially.

\begin{remark}
	We will use ``constants'' depending polynomially on $M_k$. We tacitly assume without loss of generality that all these polynomials in $M_k$ are greater than $4$ and monotonically non-decreasing.
	We did not optimize our proofs to yield the smallest possible polynomials, rather preferring clarity of exposition.
\end{remark}

Backward and mixed stability \cite[Section 1.5]{Higham1996} are defined independently of the condition number of the computational problem.

\begin{definition}[Backward stability]\label{def_backward_stability}
	Let $f : \cupdot_k S_k \to \cupdot _k \R^{n_k}$ be a scalable function and $\hat{f}$ a BSS machine. Then, $\hat{f}$ is a \emph{backward stable} algorithm for $f$ if there exist polynomials $a$ and $b$ such that for all $x \in S_{k}$,
	\[
	0 < u \leq \frac{1}{a(M_k)} \Rightarrow \exists y : \hat{f}^u(x) = f(y) \;\text{ and }\; \dist(x,y) \leq b(M_k) u.
	\]
	In particular, an output must be produced for all such choices of $x$ and $u$.
\end{definition}

\begin{definition}[Mixed stability] \label{def_mixed_stability}
	Let $f : \cupdot_k S_k \to \cupdot _k \R^{n_k}$ be a scalable function and $\hat{f}$ a BSS machine. Then, $\hat{f}$ is a \emph{mixed stable} algorithm for $f$ if there exist polynomials $a$, $b$, and $c$ such that for all $x \in S_{k}$
	\[
	0<u \leq \frac{1}{a(M_k)} \Rightarrow \exists y : \dist(\hat{f}^u(x),f(y)) \leq b(M_k) u \;\text{ and }\; \dist(x,y) \leq c(M_k) u.
	\]
	In particular, an output must be produced for all such choices of $x$ and $u$.
\end{definition}

The final classic notion of stability is forward stability, which just says that an algorithm is stable if its output $\hat f^u(x)$ is close to the exact value $f(x)$, with the caveat that this distance may depend linearly on the condition number at $x$.
\begin{definition}[Forward stability] \label{def_fs}
	Let $f: \cupdot_k S_k \to \cupdot _k \R^{n_k}$ be a scalable function and $\hat{f}$ a BSS machine. Then, $\hat{f}$ is a \emph{forward stable} algorithm implementing $f$ if there exists a \emph{stability polynomial} $a$ such that either it holds for all $x \in S_{k}$ that
	\[
	0 < u \leq \frac{1}{\tilde{\kappa}(f,x) a(M_k)} \;\Rightarrow\;  \dist(\hat{f}^u(x),f(x)) \leq a(M_k) \tilde{\kappa}(f,x)  u,
	\]
	or, equivalently, it holds for all $x \in S_{k}$ and all $\epsilon \in [0, 1]$ that
\[
0 < u \leq \frac{\epsilon}{\tilde{\kappa}(f,x) a(M_k)} \;\Rightarrow\; \dist(\hat{f}^u(x),f(x)) \leq  \epsilon.
\]
 In particular, an output must be produced for all such choices of $x$ and $u$.
If $\tmu(f,x)=\infty$ and $u\ne 0$, then the foregoing implications are vacuous and a forward stable algorithm may either output any value, or not halt and output no value at all.
\end{definition}

It it is not hard to verify that the two conditions in \cref{def_fs} are equivalent. The reason for giving two alternatives is twofold. First, each of them will be useful in our analyses. Second, they carry a different philosophy: the first criterion shows that given any sufficiently small unit roundoff $u$ a forward stable algorithm guarantees a certain accuracy, while the other shows that for every wanted accuracy $\epsilon$ one can find a unit roundoff such that a forward stable algorithms achieves the required accuracy.

The reason to use $\tmu$ instead of $\kappa$ in \cref{def_fs} is that for some problems $\kappa$ may be equal to $0$ or a very small number, and thus asking the error to depend linearly on $\kappa$ is just too much to be a realistic demand. We believe that most of the algorithms usually regarded as backward, mixed or forward stable in numerical analysis satisfy \cref{def_backward_stability}, \ref{def_mixed_stability} or \ref{def_fs}, respectively.

\section{Stability of numerical algorithms for amenable problems}\label{sec_main}

This section introduces the five main innovations of this paper. First, we present the class of \textit{amenable} computational problems. These are scalable functions that simultaneously admit a non-uniform bounded growth of the condition number and are well-behaved near the boundaries of their domains. Second, we propose a \textit{compatibility condition} for amenable functions $g$ and $h$, under which we can prove amenability of the composition $g \circ h$. Third, we prove the main theorem: composing two forward stable numerical algorithms $\hat{g}^u$ and $\hat{h}^u$ that respectively implement compatible amenable functions $g$ and $h$ results in a forward stable algorithm $\hat{g}^u \circ \hat{h}^u$ implementing the amenable function $g \circ h$. Fourth, under amenability a helpful chain of implications arises wherein backward implies mixed implies forward stability. Fifth, for a \textit{differentiable} problem $f : \R^m \to \R^n$, proving forward stability of $\hat{f}^u$ in the relative error metric reduces to establishing the stability of all component functions $\hat{f}^u_i$ for computing the $f_i$'s.

\subsection{The amenability condition}
Our main goal is to facilitate the composition of stable numerical algorithms, resulting in a new stable algorithm. Consider two scalable functions $g$ and $h$ that can be composed. We identified three obstacles that seem to prevent unbridled composition of forward stable numerical algorithms $\hat{g}^u$ and $\hat{h}^u$ implementing respectively $g$ and $h$:
\begin{enumerate}
 \item $\hat{g}^u \circ \hat{h}^u$ is no longer well-defined for some or even all $u$;
 \item the condition number of either $g$ or $h$ grows uncontrollably;
 \item the maximum of the condition numbers $\kappa(g,h(x))$ and $\kappa(h,x)$ is significantly larger than $\kappa(g \circ h,x)$.
\end{enumerate}

The issue with the first item is clear. The second obstacle may prevent the condition number from being useful since we want the condition number (a first order variation estimator) to provide reasonable bounds for moderately small values of $u$.
 The third obstacle is behind the instability of the naive method to solve an overdetermined least-squares problem $C\mathbf{x}=\mathbf{d}$ discussed in the introduction, and  was even exploited in \cite{NT2016} to prove forward instability of certain resultant-based methods to solve systems of polynomial equations, and in \cite{BBV2019} to prove forward instability of pencil-based methods for computing tensor rank decompositions.

Observe that the second and third obstacles are formulated independently of the algorithms $\hat{g}^u$ and $\hat{h}^u$. Remarkably, the first obstacle can also be avoided by placing suitable restrictions \emph{only} on the functions $g$ and $h$. The central idea of our notion of amenability is to prevent the occurrence of the first two obstacles.

\begin{definition}[Amenable function]\label{def_amenable}
	A scalable function $f$ is \emph{amenable} if there exists an \emph{amenability polynomial} $a$ such that:
	\begin{enumerate}
		\item[(A.1)]
		For all $x\in S_k$, the ball
		\[
		B_x=\left\{y\in\R^{m_k}:\dist(x,y)\leq \frac{1}{a(M_k)\tmu(f,x)}\right\}
		\]
		is contained in $S_k$.
		\item[(A.2)] For all $y\in B_x$ we have $\tmu(f,y)\leq a(M_k)\tmu(f,x)$.
	\end{enumerate}
\end{definition}
\begin{remark}\label{rmk:dimensions}
Sometimes we want to deal with functions that are not scalable.
In these cases we still talk about amenability but the polynomial $a$ becomes just a constant. The same applies if a function is defined on $\cupdot_k S_k$ where $k$ runs over a finite set.
\end{remark}

For all points $x\in S_k$ where $\tmu(f,x)=\infty$, the {two conditions of Definition \ref{def_amenable}} are automatically satisfied. Thus, it suffices to verify amenability for all points $x$ where $\tmu(f,x)$ is finite.

The following lemma will be helpful to check if a given function is amenable.

\begin{lemma}\label{lem:Gronwall}
	Given a scalable function $f : \cupdot_k S_k \to \cupdot_k \R^{n_k}$,
assume that there is a polynomial $a$ such that for all $k\in\N$ the following properties hold:
	\begin{enumerate}
		\item[(i)] Let $(x_0, x_1,x_2,\ldots)\subseteq S_k$ be a sequence such that
		\[
		\dist(x_j,\partial S_k\cup\mathcal{I}_k)\to0,
		\]
		where $\partial S_k$ is the boundary
		of $S_k$ and $\mathcal{I}_k = \{x\in S_k:\kappa(f,x)=\infty\}$ is the ill-posed locus. Then, $\kappa(f,x_j)\to\infty$.
		\item[(ii)]  Let  $V_k = S_k\setminus\mathcal{I}_k$ be the set where the condition number is finite and let $\tmu_f:V_k\to \R,\, x \mapsto \tmu(f,x)$. The condition number of $\tmu_f$ satisfies
	\begin{equation}\label{eq:alternative}
	\kappa(\tmu_f,x)\leq  \frac{a(M_k)}{4}\tmu(f,x),\quad \forall x\in S_k.
	\end{equation}
	\end{enumerate}
	Then, $f$ is amenable with amenability polynomial $a$.
\end{lemma}
\begin{proof}
See \cref{sec:P2}.
\end{proof}

\begin{remark}\label{rmk:lips}
	If $\kappa(f,x)$ is given by a smooth formula, then \eqref{eq:alternative} is satisfied if
	\begin{equation}\label{eq:alternative2}
	\sqrt{ \sum_{i} \left(x_i\frac{\partial \kappa}{\partial x_i}\right)^2 }
	\leq q(M_k)\tmu(f,x)^2,
	\end{equation}
	for some polynomial $q$.
\end{remark}

It follows immediately in the notation of \cref{lem:Gronwall} that each $V_k$ is an open subset of $\R^{m_k}$ and that the restriction $\kappa\mid_{V_k}$ is continuous for all $k\in\N$.

Note that \cref{eq:alternative2} essentially bounds the gradient of the condition number $\kappa(f,x)$ as the square of the latter. Many condition numbers in linear algebra satisfy such a bound \cite{Demmel1987,DesHigham1995}.

As said in \cref{sec:cond}, once we have endowed $\R^d$ with the coordinatewise relative metric, its induced topology has $3^d$ connected components. Fortunately, our definition of amenability takes care of that, as can be easily proved.

\begin{lemma}\label{lem:domainaditivity}
	Let $\R^d$ be endowed with any metric and consider its metric topology. If $f|_{S_\alpha}$ is amenable for a finite indexed collection of open sets $S_\alpha\subseteq\R^d$, then $f|_S$ is amenable, where
	\(
	S=\bigcup_\alpha S_\alpha.
	\)
Moreover, if all the $f|_{S_\alpha}$ admit the same amenability polynomial, then so does $f|_S$, even if $\alpha$ runs through an infinite set.
\end{lemma}

\subsection{The compatibility condition}
The third obstacle for stably composing numerical algorithms is decomposing a well-conditioned computational problem $f$ into a composition $f = g \circ h$ where either $g$ or $h$ is poorly conditioned (compared to $f$). This can lead to a forward unstable algorithm \cite{NT2016,BBV2019}. Based on this observation, we propose the following compatibility condition that excludes this possibility.

\begin{definition}[Compatibility condition] \label{def_compatible}
Let $g,h$ be scalable functions such that the composition $g \circ h$ is well defined. Graphically, for $k\geq0$
\[
\begin{matrix}
  S_k\subseteq\R^{m_k} &\stackrel{h}\to& T_k\subseteq\R^{n_k} &\stackrel{g}\to& \R^{p_k} \\
  x                    &\mapsto        &h(x)                  &\mapsto        &g(h(x)).
\end{matrix}
\]
We say that $g,h$ are \emph{compatible} if there exist \emph{compatibility polynomials} $b,c$ such that:
\begin{enumerate}
  \item For all $k$, we have $n_k \leq b(M_k)$, where $M_k=\max(m_k,p_k)$. In other words, the intermediate dimension is polynomially bounded by the maximum of the input and the output dimensions of the composition.
  \item For all $k$ and $\forall x\in S_k$ there holds
\begin{equation}\tag{C}\label{eq:reves}
\tmu(g,h(x))\tmu(h,x) \leq c(M_k) \tmu(g\circ h,x).
\end{equation}
That is, \cref{eq:boundmu} can be reverted up to a factor which is polynomial in $M_k$.
\end{enumerate}
\end{definition}

The compatibility condition guarantees that the composition of two amenable functions is again amenable, as is shown next.

\begin{proposition} \label{prop_compose_amenable}
 If $g$ and $h$ are compatible amenable functions, then $g \circ h$ is also amenable.
\end{proposition}

\begin{proof}
Let $a_g$ and $a_h$ be the amenability polynomials of respectively $g$ and $h$, and let $b,c$ be the polynomials appearing in the definition of compatibility. We will prove that $f=g\circ h$ is amenable with amenability polynomial $a(t) = a_g(b(t)) a_h(b(t)) c(t)$. Without loss of generality, we can assume that $b(t)\geq t$.

For any given $k$, the input and output dimensions of a function $\varphi$ are denoted by $m_k^{\varphi}, n_k^\varphi$ respectively, and $M_k^\varphi=\max \{m_k^{\varphi}, n_k^\varphi \}$. We thus have $M_k^h,M_k^g\leq b(M_k^f)$ from the compatibility, which implies
\[
a(M_k^f)\geq a_g(M_k^g)a_h(M_k^h)c(M_k^f),
\]
having exploited that all polynomials are non-decreasing by assumption.
It suffices to verify amenability for all $x\in S_k$ such that $\tmu(g\circ h,x)$ is finite. It follows from compatibility that both $\tmu(h,x)$ and $\tmu(g, h(x))$ are finite.

	Let $x\in S_k$, let $y\in \R^{m_k}$ be such that
	\[
	\dist(x,y)
	\leq \frac{1}{a(M_k^f)\tmu(f,x)}
	\stackrel{\eqref{eq:reves}}{\leq} \frac{c(M_k^f)}{a(M_k^f)\tmu(h,x)}
	\leq \frac{1}{a_h(M_k^h)\tmu(h,x)}.
	\]
	Then, $y\in S_k$ by using (A.1) for $h$. Hence, $y$ is in the domain of $f$ and (A.1) for $f$ holds. Moreover, $\tmu(h,y)\leq a_h(M_k^h)\tmu(h,x)$ by (A.2) for $h$. Similarly, for all $y\in B_x$ it follows from \cref{lem:geodesicas} that
	\begin{align*}
	\dist(h(x),h(y))
	&\leq \dist(x,y) \cdot \max_{z \in B_x}\tmu(h,z)\\
	&\leq\dist(x,y)a_h(M_k^h)\tmu(h,x)
	\leq \frac{a_h(M_k^h) \tmu(h,x)}{a(M_k^f)\tmu(f,x)}
	\stackrel{\eqref{eq:reves}}\leq\frac{1}{a_g (M_k^g) \tmu(g,h(x))},
	\end{align*}
	so that $\tmu(g,h(y))\leq a_g(M_k^g)\tmu(g,h(x))$ by (A.2) for $g$. Consequently,
	\[
	\tmu(f,y)
	\stackrel{\eqref{eq:boundmu}}{\leq} \tmu(h,y)\tmu(g,h(y))
	\leq a_h(M_k^h) a_g (M_k^g)\tmu(h,x)\tmu(g,h(x))
	\stackrel{\eqref{eq:reves}}{\leq} a(M_k^f) \tmu(f,x).
	\]
	Thus (A.2) holds for $f$ as well, concluding the proof.
\end{proof}

\subsection{The composition theorem} \label{sec:fundamental}

Now we can prove the main result of this article on the composition, or concatenation, of forward stable numerical algorithms in the setting of compatible amenable computational problems.

\begin{theorem}
\label{th:composition}
Let $g$ and $h$ be compatible amenable functions. For all forward stable algorithms $\hat{g}^u$ and $\hat{h}^u$ implementing respectively $g$ and $h$, the composition $\hat{g}^u \circ \hat{h}^u$ is a forward stable algorithm implementing $f = g \circ h$.
\end{theorem}
\begin{proof}
Let $h : \cupdot_k S_k \to \cupdot_k T_k$ and $g : \cupdot_k T_k \to \cupdot_k \R^{p_k}$ with $S_k \subset \R^{m_k}, T_k\subseteq \R^{n_k}$. Denote $M_k^h=\max\{m_k,n_k\}$, $M_k^g=\max\{n_k,p_k\}$ and $M_k^f=\max\{m_k,p_k\}$.
Let $a_g(t)$ and $a_h(t)$ be the amenability polynomials from \cref{def_amenable} for $g$ and $h$, respectively. Let $\hat{a}_g(t)$ and $\hat{a}_h(t)$ be the stability polynomials from \cref{def_fs} of $\hat g$ and $\hat h$, respectively. Finally, let $b(t),c(t)$ be the compatibility polynomials from \cref{def_compatible} for $g$ and $h$, which by compatibility implies $M_k^h,M_k^g\leq b(M_k^f)$.
 We will show that the theorem holds with stability polynomial
 \[
 a(t) = c(t) a_g(b(t)) \bigl( \hat{a}_g(b(t)) + \hat{a}_h(b(t)) \bigr).
 \]
	
	When $\tmu(g\circ h,x)=\infty$, there is nothing to check. Thus, in the remainder of the proof we can assume by compatibility that both $\tmu(g,h(x))$ and $\tmu(h,x)$ are finite.
	
	Since $g$ and $h$ are compatible and $\hat{h}^u$ is forward stable, we have
	\begin{multline}
	u\leq\frac{1}{a(M_k^f) \tmu(f,x)}
	\stackrel{\eqref{eq:reves}}
	\leq \frac{c(M_k^f)}{a(M_k^f)\tmu(h,x)}
	\leq\frac{1}{\hat{a}_h(M_k^h) \tmu(h,x)} \Rightarrow \\
	\dist(\hat{h}^u(x),h(x))\leq \hat{a}_h(M_k^h) \tmu(h,x)u
	\stackrel{\cref{eq:reves}}{\leq} \frac{c(M_k^f) \hat{a}_h(M_k^h)}{a(M_k^f)\tmu(g,h(x))}
\leq \frac{1}{a_g(M_k^g) \tmu(g,h(x))}
	\label{eq:1},
	\end{multline}
	where in the second use of \cref{eq:reves} the consequent $u$ was first bounded by $\frac{1}{a(M_k^f) \tmu(f,x)}$.
 We know that the ball $B_{h(x)}$ with radius $\frac{1}{a_g(M_k^g) \tmu(g, h(x))}$ centered at $h(x)$ is contained in $T_k$.
	Consequently, $g(\hat{h}^u(x))$ is well defined and a minimizing geodesic $\gamma_h:[0,1]\to T_k$ from $h(x)$ to $\hat{h}^u(x)$ exists.
	Moreover, from (A.2) for $g$ it follows that
	\begin{align}\label{eqn_another}
	\tmu(g, \gamma_h(t)) \le a_g(M_k^g) \tmu(g, h(x)),\quad t \in [0,1].
	\end{align}
From \cref{lem:geodesicas}, we have
	\begin{align}
	\dist(g(\hat{h}^u(x)), g(h(x)))
	\nonumber&\leq \dist(\hat{h}^u(x),h(x)) \cdot \max_{t\in[0,1]} \tmu(g, \gamma_h(t)) \\
	\nonumber&\leq a_g(M_k^g) \tmu(g,h(x)) \dist(\hat{h}^u(x),h(x)) \\
	\nonumber&\stackrel{\eqref{eq:1}}{\leq} a_g(M_k^g) \hat{a}_h(M_k^h) \tmu(g,h(x))\tmu(h,x)u \\
	\label{eq:2} &\stackrel{\eqref{eq:reves}}{\leq} a_g(M_k^g) \hat{a}_h(M_k^h)c(M_k^f) \tmu(f,x) u.
	\end{align}
	
	From \cref{eqn_another} above, $\tmu(g,\hat{h}^u(x))\leq a_g(M_k^g) \tmu(g,h(x))$. This implies
	\[
	u
	\leq \frac{1}{a(M_k^f) \tmu(f,x)}
	\leq \frac{c(M_k^f)}{a(M_k^f)\tmu(g,h(x))}
	\leq \frac{a_g(M_k^g) c(M_k^f)}{a(M_k^f)\tmu(g,\hat{h}^u(x))}
	\leq  \frac{1}{\hat{a}_g(M_k^g) \tmu(g,\hat{h}^u(x))}.
	\]
	From the forward stability of $\hat{g}^u$, it follows that $\hat{g}^u(\hat{h}^u(x))$ is well defined and
	\begin{equation}
	\dist(\hat{g}^u(\hat{h}^u(x)),g(\hat{h}^u(x)))
	\leq \hat{a}_g(M_k^g) \tmu(g,\hat{h}^u(x)) u
	\leq a_g(M_k^g) \hat{a}_g(M_k^g) \tmu(g,h(x)) u.\label{eq:3}
	\end{equation}
	
	Putting the two bounds \cref{eq:2,eq:3} together, we have proved that
	\begin{align*}
	\dist(\hat{f}^u(x),f(x))
	&\leq \dist(\hat{g}^u(\hat{h}^u(x)),g(\hat{h}^u(x)))+\dist(g(\hat{h}^u(x)),g(h(x)))\\
	&\leq a_g(M_k^g)\hat{a}_g(M_k^g)\tmu(g,h(x))u+ a_g(M_k^g) \hat{a}_h(M_k^h)c(M_k^f)\tmu(f,x)u\\
	&\leq (a_g(M_k^g) \hat{a}_g(M_k^g) c(M_k^f) + a_g(M_k^g) \hat{a}_h(M_k^h) c(M_k^f))\tmu(f,x)u.
	\end{align*}
	The theorem follows using again that $M_k^g,M_k^h \leq b(M_k^f)$.
\end{proof}

\begin{remark}
One can extend \cref{th:composition} to the composition of three or more maps by applying it repeatedly. For example, if the composition $f\circ g\circ h$ makes sense and the three functions are amenable, one must check that $g$ and $h$ are compatible, and that $f$ and $g\circ h$ are compatible. Then, given stable algorithms $\hat f,\hat g, \hat h$, from \cref{th:composition} we have a stable algorithm $\hat g\circ \hat h$ for $g\circ h$ and again from \cref{th:composition} the algorithm $\hat f\circ \hat g\circ \hat h$ is stable for $f\circ g\circ h$.
\end{remark}

\subsection{A chain of implications}

The previous subsection introduced a potent tool for checking the forward stability of numerical algorithms for amenable problems: if the algorithm can be decomposed as a composition of two forward stable numerical algorithms whose functions are amenable and compatible, then the original algorithm is forward stable.
Now we derive a chain of implications between the notions of stability introduced in \cref{sec_numerical_stability}. This yields another powerful tool for proving forward stability of numerical algorithms for amenable problems. In particular, it will be shown that both backward and mixed stable algorithms are also forward stable under the hypothesis of amenability. Since the presented notions of backward and mixed stability are essentially the classic ones, we postulate that most known backward and mixed stable algorithms in the literature for amenable problems are forward stable in the formal sense of \cref{def_fs}.

We start by observing that backward stability implies mixed stability.

\begin{proposition} \label{thm_backward_implies_mixed}
	Let $f$ be scalable. If $\hat{f}^u$ is a backward stable numerical algorithm implementing $f$, then it is mixed stable.
\end{proposition}

The reverse implication does not hold. For example, consider the function that maps a matrix to the orthogonal matrix in the $QR$ factorization. A standard implementation of this method can be proved to be mixed stable; however, due to the fact that the output cannot be guaranteed to be orthogonal for finite $u$ as irrational numbers are not representable in $\F_u$, it is not backward stable.

When the underlying function $f$ is amenable, mixed implies forward stability.

\begin{theorem} \label{thm_mixed_implies_fp}
	Let $f$ be amenable. If $\hat{f}^u$ is a mixed stable numerical algorithm implementing $f$, then it is forward stable.
\end{theorem}
\begin{proof}
    Let $a(t)$, $b(t)$ and $c(t)$ be the polynomials from \cref{def_mixed_stability} for $\hat{f}^u$. Let $d(t)$ be the amenability polynomial in \cref{def_amenable} for $f$. We now show that $\hat{f}^u$ is forward stable with stability polynomial $p(t) = d(t) c(t) + a(t) + b(t)$.

    We can assume $\tmu(f,x) < \infty$, for otherwise forward stability is vacuously satisfied. Let $u \leq \frac{1}{p(M_k) \tilde{\kappa}(f,x)}$. By the triangle inequality, we have
	\[
	\dist(\hat{f}^u(x),f(x)) \leq \dist(\hat{f}^u(x),f(y)) + \dist(f(y),f(x)),
	\]
	where $y$ is as in \cref{def_mixed_stability}.
	The first term is bounded by $b(M_k) u$ by definition of mixed stability. To estimate the second addend, we rely on \cref{lem:geodesicas} to get
	\[
	\dist(f(y),f(x)) \leq \dist(x,y) \cdot \max_{t\in[0,1]} \kappa(f, \gamma(t)),
	\]
	where $\gamma(t)$ is the minimizing geodesic from $x$ to $y$. It exists from (A.1) for $f$, since
	\[
	 \dist(x,y) \le c(M_k) u \le \frac{c(M_k)}{p(M_k) \tmu(f,x)} \le \frac{1}{d(M_k) \tmu(f,x)}.
	\]
    By (A.2) for $f$, we obtain
    \[
     \dist(f(y),f(x)) \leq \dist(x,y) \cdot \max_{t\in[0,1]} \kappa(f, \gamma(t)) \le  d(M_k) c(M_k) \tmu(f,x) u.
    \]
    By putting everything together we obtain
    \[
     \dist(\hat{f}^u(x),f(x)) \leq b(M_k) u + d(M_k) c(M_k) \tmu(f,x) u \le p(M_k) \tmu(f,x) u,
    \]
    which concludes the proof.
\end{proof}

\subsection{The stacking theorem} \label{sec_stacking_n_functions}
The previous subsection showed how existing results combined with amenability can be leveraged to prove stability. Now we show our tool to extend claims on univariate functions to vector, matrix and tensor--valued mappings.

\begin{theorem}\label{prop:stacking}
Let  $f : S_k\subseteq\R^{m_k} \to \R^{n_k}$ be a scalable function. Recall that $M_k=\max(m_k,n_k)$, and assume that there exists a polynomial $a(t)$ such that:
\begin{enumerate}
\item Each coordinate function $f_{k,i}:S_k\to \R$, $1\leq i\leq n_k$, is, viewed as a non-scalable function, amenable with constant amenability polynomial $a(M_k)$.
\item There exist corresponding forward stable numerical algorithms $\hat{f}_{k,i}^u$ for $f_{k,i}$ with constant stability polynomial $a(M_k)$.
\end{enumerate}
Then, $f$ is amenable as a scalable function with amenability polynomial $A(t)=ta(t)$ and the stacked algorithm $\hat{f}^u$ given by $\hat {f}^u_k(x)= (\hat{f}_{k,1}^u(x), \ldots, \hat{f}_{k,n_k}^u(x))$ is a forward stable algorithm for $f$ with stability polynomial $A(t)$.

\end{theorem}
\begin{proof}
The condition number of $f_k=f|_{S_k}$ and its component functions $f_{k,i}$ are related as follows:
\begin{align} \label{eqn_important_stacking_bound}
\max_{1\leq i\leq n }\kappa(f_{k,i},\vect{x})
\leq\kappa(f_k, \vect{x})
\leq \sqrt{\sum_{i=1}^{n_k} \kappa(f_{k,i},\vect{x})^2}
\le \sqrt{n_k} \max_{1\leq i\leq n }\kappa(f_{k,i},\vect{x}).
\end{align}
That is, the condition number of the stacked problem for any fixed $k$ is proportional to the condition number of the worst function in the stacking for that fixed $k$.
This is indeed easy to check from the definition.

We check directly that $f$ is amenable from \cref{def_amenable}. Fix any $k\geq1$ and $\vect{x}\in S_k$. Since all $f_{k,i}$ are amenable with constant polynomial $a(M_k)$, we have
\[
\dist(\vect{x},\vect{y}) \leq \frac{1}{a(M_k) \tmu(f_{k,i},\vect{x})} \Rightarrow \begin{cases}\vect{y}\in S_k\\\tmu(f_{k,i},\vect{y})\leq a(M_k) \tmu(f_{k,i},\vect{x}).\end{cases}
\]
Then, for all $i=1,\ldots,n$, with $A(t) = ta(t)$ we have
\begin{multline*}
\dist(\vect{x},\vect{y}) \leq \frac{1}{A(M_k) \tmu(f_k,\vect{x})} \Rightarrow \dist(\vect{x},\vect{y}) \leq \frac{1}{a(M_k) \tmu(f_{k,i},\vect{x})}\Rightarrow\\ \tmu(f_{k,i},\vect{y})\leq a(M_k) \tmu(f_{k,i},\vect{x}),
\end{multline*}
and hence the last item in \cref{def_amenable} is easily checked:
\begin{multline*}
\tmu(f_k,\vect{y})\leq \sqrt{\sum_{i=1}^{n_k} \tmu(f_{k,i},\vect{y})^2}
\leq a(M_k) \sqrt{\sum_{i=1}^{n_k}\tmu(f_{k,i},\vect{x})^2}
\leq \\\sqrt{n_k} a(M_k) \tmu(f,\vect{x})\leq A(M_k)\tmu(f,\vect{x}).
\end{multline*}
We have thus proved that each $f_k$ is amenable with constant $A(M_k)$, so $f$ is amenable with polynomial $A(t)$.

Assume now that the functions $f_{k,i} : S_k \to \R$ have forward stable algorithms $\hat{f}_{k,i}^u$. Fix any $k\geq1$. By \cref{def_fs}, and reusing the amenability polynomials $a$ and $A$, we have
\[
\forall \vect{x} \in S_k, \forall u < \frac{1}{a(M_k) {\tmu}(f_{k,i}, \vect{x})}: \dist(\hat{f}_{k,i}^u(\vect{x}), f_{k,i}(\vect{x})) \le a(M_k) {\tmu}(f_{k,i}, \vect{x}) \, u.
\]
Hence,
\begin{multline*}
 \dist(\hat{f}_k^u(\vect{x}), f_k(\vect{x})) = \sqrt{\sum_{i=1}^{n_k} \dist(\hat{f}_{k,i}^u(\vect{x}), f_{k,i}(\vect{x}))^2} \le\\  a(M_k)  \, u \sqrt{\sum_{i=1}^{n_k} \kappa(f_{k,i},\vect{x})^2}\le A(M_k) {\tmu}(f_k,\vect{x}) \, u,
\end{multline*}
proving that $\hat{f}^u_k$ is forward stable for $f_k$ with stability constant polynomial $A(M_k)$, that is, $\hat f^u$ is forward stable for $f$ with stability polynomial $A(t)$.
\end{proof}

Notable cases include replicating an input $n$ times and the identity map on $\R^n$.

\section{Amenable problems and stable algorithms to solve them}
\label{sec:amenandstable}

We now describe a catalogue of elementary but fundamental problems that are amenable when we endow our input and output spaces with the coordinatewise relative distance. We also provide forward stable algorithms for solving them.

The next graph summarizes the results we establish in the next subsections:\\[-5pt]
\begin{center}\begin{tikzpicture}[every node/.style={text width=3cm, text centered, font=\footnotesize}]
\node[draw] (A) {6.1. Elementary\\ functions};
\node[draw,below=.2cm of A] (B) {6.2. Multiplication};
\node[draw,below=.2cm of B] (C) {6.2. Summation};
\node[draw,right=.6cm of B] (D) {6.3. Arithmetic on\\ tensors};
\node[draw,right=4.5cm of A] (E) {6.5. Euclidean inner product};
\node[draw,right=4.5cm of C] (F) {6.4. Linear maps};
\node[draw,right=4.5cm of B] (G) {6.5. Euclidean norm};
\draw[->,>=stealth,thick] (A) -- ++(1.75cm,0) |- (D);
\draw[thick] (B) -- (D);
\draw[thick] (C) -- ++(1.75cm,0) |- (D);
\draw[->,>=stealth,thick] (D) -- (G);
\draw[->,>=stealth,thick] (D) -- ++(1.75cm,0) |- (E);
\draw[->,>=stealth,thick] (D) -- ++(1.75cm,0) |- (F);
\end{tikzpicture}\end{center}
This section can be skipped on a first reading, as its main goal is to establish amenability of the above problems, and prove the stability of basic algorithms for solving them.

\subsection{Elementary functions} \label{sec_fundamental_constructions}
It is verified immediately that the identity map $\mathrm{Id} : \R \to \R,\; x \mapsto x$, the constant map $\alpha : \R \to \R,\; x \mapsto \alpha$, and the inversion map $\cdot^{-1}: \R_0 \to \R_0,\; x \mapsto x^{-1}$ have constant condition numbers bounded by $1$.
Moreover, since they are defined everywhere (except possibly at $0$), the conditions of \cref{def_amenable} are satisfied, so they are amenable.
The algorithms implementing these functions as stated are forward stable because in all cases
\(
 \dist(\hat{f}^u(x), f(x)) = \log \left| \frac{f(x)(1+\delta)}{f(x)} \right| \le 2u
\)
where $|\delta| \le u < 1/4$.

Let $\alpha \in \R$ be a constant. It can be verified that $x \mapsto x \circ \alpha$, where $\circ \in \{+, -, \cdot\}$, has a forward stable algorithm that consists of replacing $\circ$ by the floating-point implementation $\hat{\circ}^u$ (use \cref{lem:Gronwall} for amenability). For division $x \mapsto x/\alpha$ the same holds if $\alpha$ is nonzero and likewise for $x \mapsto \alpha/x$ with $x$ nonzero.

\subsection{Multiplication and summation}\label{sec:muliandsum}
The scalable function codifying the product,
\(
\Pi : \mathbf{x}=(x_1, \ldots, x_k) \mapsto x_1 \cdots x_k
\)
on the domain $D = \cupdot_k \R^k$, has condition number
\[
\kappa(\Pi,\vect{x}) = \left\| \begin{bmatrix} \frac{x_1 \tfrac{\partial}{\partial x_1} (x_1 \cdots x_k)}{x_1 \cdots x_k} & \cdots & \frac{x_k \tfrac{\partial}{\partial x_k} (x_1 \cdots x_k)}{x_1 \cdots x_k} \end{bmatrix} \right\|_2
= \sqrt{k},
\]
if all the $x_i$ are different from $0$, and otherwise
\(
\kappa(\Pi,\vect{x}) = 0
\)
because $\Pi$ is locally constant. As this is constant in the $x_i$'s in both cases, we obtain $\kappa(\kappa_\Pi,\vect{x}) = 0$, so $\Pi$ satisfies the hypotheses of \cref{lem:Gronwall}. Hence, the function is amenable in all of the connected components of $\R^k$ (with the topology of relative error), which from \cref{lem:domainaditivity} implies that $\Pi$ is amenable in $\R^k$ and also in $D$.

It is easy to see from the definition that the trivial algorithm for $\Pi$ that generates the sequence $x_1,x_1x_2,\ldots,\Pi(\vect{x})=x_1\cdots x_k$ is backward stable: its output is
\(\Pi(\vect{x})\prod_{i=1}^{2k-1}(1+\delta_i)\) where $\delta_i\in(-u,u)$, that is the exact product of $y_1,x_2,\ldots,x_k$ where the relative distance from $y_1$ to $x_1$ is at most $|\log(1-u)^k|\leq 2ku$ whenever $u<1/(4k)$. Hence, from \cref{def_backward_stability} with stability polynomials $a(k)=b(k)=4k$, the algorithm is backward stable and from \cref{thm_mixed_implies_fp} we deduce that it is forward stable.

The same strategy shows that, for any fixed $k\in\Z$, the map $x\to x^k$ is amenable (once we remove the obvious exceptions such as $0^{-1}$) and the straightforward algorithm is forward stable.

A similar argument works with summation, which also defines a scalable function that we denote by $\Sigma : \cupdot_k \R^k \to \R,\; (x_1, \ldots, x_k) \mapsto x_1 + \cdots + x_k$. The condition number can be computed from \cref{prop_cond_expression}(i):
\[
\kappa(\Sigma,\vect{x}) = \left| \begin{bmatrix} \frac{x_1 \tfrac{\partial}{\partial x_1}(x_1 + \cdots + x_k)}{x_1 + \cdots + x_k} & \cdots & \frac{x_k \tfrac{\partial}{\partial x_k}(x_1 + \cdots + x_k)}{x_1 + \cdots + x_k} \end{bmatrix} \right\|_2 = \frac{\sqrt{\sum_{i=1}^k x_i^2}}{|x_1 + \cdots + x_k|}=\frac{\|x\|_2}{|\Sigma(x)|},
\]
if $\Sigma(\vect{x})\ne0$. When $\Sigma(\vect{x})=0$ and $\vect{x}\ne0$, \cref{prop_cond_expression}(ii) entails that $\kappa(\Sigma,\vect{x})=\infty$, so the expression {above} is valid for all $\vect{x}\ne0$. It is a routine exercise to show that
\[
\kappa(\kappa_\Sigma,\vect{x})\leq 2k \kappa(\Sigma,\vect{x}),
\]
showing again that $\Sigma$ is amenable. As in the case of multiplication, it is easy to see that the naive algorithm for $\Sigma$ is backward stable which again under amenability implies forward stable from \cref{thm_mixed_implies_fp}.

\subsection{Arithmetic operations on vectors, matrices and tensors}\label{sec_tensor_arithmetic_amenable}

Since we have seen that stacking amenable functions gives another amenable function, it follows immediately that the following maps with domain $\R^{n_1 \times \cdots \times n_d}$ are amenable functions (recall \cref{rmk:dimensions} on the role of the dimensions of the input and range):
\begin{enumerate}
 \item multiplication by a constant, i.e., $X \mapsto \alpha X$ for $\alpha \in \R$.
 \item adding a constant, i.e., $[ X_{i_1,\ldots,i_d} ] \mapsto [X_{i_1,\ldots,i_d} + \alpha]$ where $\alpha \in \R$.
 \item adding a constant array, i.e., $X \mapsto X + A$ with $A \in \R^{n_1 \times \cdots \times n_d}$.
 \item adding $n$ arrays, i.e., $(X_1, \ldots, X_n) \mapsto X_1 + \cdots + X_n$, where $X_i \in \R^{n_1\times\cdots\times n_d}$.
 \item \textit{Hadamard product} with a constant array $A \in \R^{n_1 \times \cdots \times n_d}$, i.e.,  $X \mapsto X \circledast A = [X_{i_1,\ldots,i_d} A_{i_1,\ldots,i_d}]$;
 \item Hadamard product of $n$ arrays, i.e., $(X_1, \ldots, X_n) \mapsto X_1 \circledast \cdots \circledast X_n$.
\item \textit{tensor product} with a constant array $A \in \R^{m_1 \times \cdots \times m_e}$, i.e.,
\(
 X \mapsto X \otimes A = [X_{i_1,\ldots,i_d} A_{j_1,\ldots,j_e}],
\)
where the image is an element of $\R^{n_1 \times \cdots \times n_d \times m_1 \times \cdots \times m_e}$.
\item tensor product of $n$ arrays, i.e., $(X_1, \ldots, X_n) \mapsto X_1 \otimes \cdots \otimes X_n$ (this time the domain is $(\R^{n_1 \times \cdots \times n_d})^n$). As a special case, the \textit{Kronecker product} of two matrices is amenable.
\end{enumerate}

For all aforementioned functions, the straightforward algorithms implementing the formulas are forward stable by \cref{prop:stacking}.

\subsection{Linear maps} \label{sec_linear_map_amenable}
We consider the sequence of linear maps $A_k : \R^{k} \to \R^{n_k}$ given in coordinates by a sequence of $n_k \times k$ matrices $A_k$.
Consider first the case $n_k = 1$ for all $k$, so that the map is the scalable function
\(
 f : \bigcupdot_k \R^k \to \R : \vect{x} \mapsto A_k \vect{x} = \sum_{i=1}^k a_i x_i
\).\footnote{To be precise, we should rather write $a_{k,i}$ since the coefficients of the matrix depend on $k$ but we think there is no ambiguity and use the less cumbersome notation.}
As with summation, we note that if one $x_i=0$, then $f_k$ is equivalent to $f_{k-1}$ where the argument $x_i$ is dropped. As $\kappa(f,0)=0$ by definition, it suffices to treat the case where all elements of $\vect{x}$ are nonzero. In this case,
the condition number is verified to be \(\kappa(f_k, \vect{x}) = \frac{1}{|A_k \vect{x}|} \left\| \begin{bmatrix} a_1 x_1 & \cdots & a_k x_k \end{bmatrix} \right\|_2\), a formula that holds even if $A_k \vect{x}=0$. We realize this map as the composition of $g(\vect{x}) = (a_1 x_1, \ldots, a_k x_k)$ and the addition map $\Sigma_k$. Note that $g$ is the Hadamard product of $\vect{x}$ with the vector of constants $\vect{a}=(a_1, \ldots, a_k)$, so it is amenable and the straightforward algorithm is stable. We have
\[
 \kappa(g,\vect{x}) = \sqrt{\chi(\vect{x})} \le \sqrt{k}
 \quad\text{and}\quad
 \kappa(\Sigma_k, \vect{x}\circledast\vect{a} ) = \frac{1}{|A_k \vect{x}|}  \left\| \vect{x}\circledast\vect{a} \right\|_2,
\]
hence showing that
\(
 \kappa(f, \vect{x}) = \kappa(\Sigma_k, \vect{x}\circledast\vect{a} ).
\)
Thus, \cref{th:composition} holds for compositions with the stable summation algorithm $\hat{\Sigma}^u$ from \cref{sec:muliandsum} with $a(t) = t+4$. This concludes the argument for $n=1$. For general $n > 1$, $f$ is amenable by stacking the $n$ amenable component functions. The corresponding algorithm is forward stable.

Note that arbitrary linear combinations of a fixed set of vectors can be computed stably using the above algorithm.

\subsection{Inner product and Euclidean norm}
The final elementary example from linear algebra we consider is computing the inner product
\[
 \langle \cdot, \cdot \rangle : \bigcupdot_k \underbrace{\R^k \times \R^k}_{S_k} \to \R, \quad (\vect{x}, \vect{y}) \mapsto \vect{x}^T \vect{y},
\]
and its induced Euclidean norm $\|\vect{x}\|_2 = \sqrt{\langle\vect{x},\vect{x}\rangle}$.
This inner product can be realized by composing $\Sigma_k$ with the Hadamard product $\circledast_k$, both of which are amenable and have stable algorithms.
Again, if either $x_i=0$ or $y_i=0$, then $\langle\cdot,\cdot\rangle_k$ is equivalent to $\langle\cdot,\cdot\rangle_{k-1}$ where both arguments $x_i$ and $y_i$ are dropped. It thus suffices to treat the case where all elements of $\vect{x}$ and $\vect{y}$ are nonzero.
As we have
\(
 \kappa(\circledast_k, (\vect{x},\vect{y})) = \sqrt{2k}
\)
and, by \cref{prop_cond_expression}(i),
\[
 \kappa(\langle \cdot, \cdot \rangle, (\vect{x},\vect{y}))
 = \frac{1}{|\vect{x}^T \vect{y}|} \| (x_1 y_1, \ldots, x_k y_k) \|_2
 = \frac{\| \vect{x} \circledast \vect{y} \|_2}{| \Sigma( \vect{x} \circledast \vect{y}) |} = \kappa(\Sigma, \vect{x} \circledast \vect{y}),
\]
(when $\vect{x}^T \vect{y}=0$, the same formula holds),
compatibility of $\langle \cdot, \cdot \rangle = \Sigma \circ \circledast$ follows. Hence by the main theorem, the composition of stable algorithms for summation and Hadamard product is a stable algorithm for the standard inner product.

The induced norm can be computed via the three-part composition
\[
 \|\vect{x}\|_2 = \sqrt \circ \left(\langle \cdot, \cdot \rangle \circ (\vect{x}\mapsto(\vect{x},\vect{x})) \right) = \sqrt \circ (\vect{x} \mapsto \| \vect{x} \|_2^2).
\]
As before, taking the norm in which some elements are zero is the same as taking the norm of the nonzero part of the vector. Hence, we can assume that all $x_i\ne0$.
First we show that $\langle \cdot, \cdot \rangle \circ (\vect{x}\mapsto(\vect{x},\vect{x}))$ is a stable algorithm for $\vect{x} \mapsto \|\vect{x}\|^2_2$. We already know the maps are amenable, so by \cref{th:composition} checking compatibility suffices. They are also compatible because
\(
 \kappa(\|\cdot\|^2_2, \vect{x}) = \frac{1}{\|\vect{x}\|_2^2} 2 \| \vect{x} \|_2^2 = 2
\)
and the copy map $c_k : \R^k \to \R^{2k},\; \vect{x} \mapsto (\vect{x},\vect{x})$ has $\kappa(c, \vect{x}) = \|[ I_{k} \; I_{k} ]\|_2 = \sqrt{2}$. The difficult one is
\[
 \kappa(\langle \cdot, \cdot \rangle, (\vect{x},\vect{x})) = \frac{\| \vect{x} \circledast \vect{x} \|_2}{ \| \vect{x} \|_2^2 } \ge \frac{1}{\sqrt{k}} \frac{\| \vect{x} \circledast \vect{x} \|_1}{ \| \vect{x} \|_2^2 } = \frac{1}{\sqrt{k}}.
\]
Compatibility follows, so this is a stable algorithm for computing the squared norm.

Amenability of $\sqrt{\cdot}$ on $\R_{\ge0}$ follows immediately from the fact that
\(
 \kappa(\sqrt{\cdot}, x)
 = \frac{1}{2}
\)
is constant and the distance to zero is $\infty$. It can be deduced from \cite[p.~113]{Ypma1983} and classic roundoff error analysis that the square root of $4^{k-1} \le g \le 4^{k}$ can be computed mixed stably by the Babylonian method; that is, Newton's method applied to $F(x) = x^2 - g 4^{-k} = 0$ starting from $\frac{1}{2}$, terminated as discussed on \cite[p.~117]{Ypma1983}, and multiplying the result with $2^k$. Forward stability follows by amenability and \cref{thm_mixed_implies_fp}. Hence, for applying \cref{th:composition}, we should check compatibility of $\sqrt{\cdot}$ and $\|\cdot\|^2_2$. For their composition, $\|\cdot\|$, we find, if $\vect{x}\ne0$,
\[
 \kappa(\|\cdot\|, \vect{x}) = \frac{1}{\|\vect{x}\|^2_2} \| (2 x_1^2, \ldots, 2 x_k^2)  \|_2 = 2\frac{\| \vect{x} \circledast \vect{x}\|_2}{\|\vect{x}\|^2_2} \le 2\frac{\| \vect{x} \circledast \vect{x}\|_1}{\|\vect{x}\|^2_2} = 2.
\]
Consequently, compatibility holds, proving that a composition of stable algorithms for $c$, $\langle\cdot,\cdot\rangle$, and $\sqrt{\cdot}$ always yields a forward stable numerical algorithm for computing the $2$-norm of a nonzero vector. If the input is $\vect{x}=0$, the algorithm must output $0$, which is clearly forward stable.

\section{A sinful function} \label{sec_sinful}
The final example we discuss is that of the innocuously-looking sine function in the relative error metric. Let us consider first the finite domain $\Omega = [k_1 \pi,k_2 \pi]$ where $k_1, k_2 \in \N$ and verify amenability. The condition number is
\[
 \kappa(\sin, x) = \left| x \frac{\cos{x}}{\sin{x}} \right|
\]
when $x \ne 0$ and $0$ by definition otherwise.
We will apply \cref{lem:Gronwall}, to prove the amenability of $\sin$. Since
$\lim_{x\to k\pi} \kappa(\sin,x) = \infty$ for all $k\in\N\setminus\{0\}$, the first condition of \cref{lem:Gronwall} holds. Recall that the distance to $0$ is $\infty$ in the relative error metric, so whatever happens to $\kappa(\sin,x)$ near zero is irrelevant for amenability. For the second condition, in the variant \cref{eq:alternative2}, we compute that
\[
 {\left| x \frac{\partial \kappa_{\sin}}{\partial x} \right| = \left| x\frac{\cos{x}}{\sin{x}} - \frac{x^2}{\sin^2 x} \right| \le \left| x\frac{\cos{x}}{\sin{x}} \right| + \frac{x^2}{\sin^2 x}.}
\]
When $x$ is an integer multiple of $\pi$, {both the foregoing and $\tmu(\sin,x)^2$ equal $\infty$.} As the singularities of $|x\partial \kappa_{\sin} / \partial x|$ and $\tmu(\sin,x)^2$ occur only at the roots of $\sin{x}$, we can multiply both by $\sin^2 x$ and subtract $|x \cos{x} \sin{x}|$ to find
\[
 \sin^2 x {\left| x \frac{\partial \kappa_{\sin}}{\partial x} \right|} - |x \cos{x} \sin{x}|
 \le x^2.
\]
On the other hand,
\[
 C \sin^2 x {\tmu(\sin,x)^2} - |x \cos{x} \sin{x}| = C \sin^2 x + (2C-1) |x \cos{x}\sin{x}| + C x^2 \cos^2{x}.
\]
As $\sin^2 x \ge \frac{1}{2}$ whenever $\cos^2 x \le \frac{1}{2}$, it follows that taking $C = 2(\max\{|k_1|, |k_2|\}\pi)^2$ ensures
\(
{|x\partial \kappa_{\sin} / \partial x|} \le C \tilde{\kappa}(\sin,x)^2,
\)
so that \cref{lem:Gronwall} holds {because of \cref{rmk:lips}}. Thus $\sin$ is amenable on any domain of the form $[k_1 \pi, k_2 \pi]$ with $k_1, k_2\in\N$.

By contrast, considering $\sin$ on $\R$ does \textit{not} result in an amenable function.
The reason is that the condition number $\kappa(\sin,x)$ is rapidly oscillating between $0$ at $\frac{\pi}{2} + k\pi$ and $\infty$ at $k\pi$. If $x$ is a point where $\kappa(f,x)=0$, then the ball $B_x$ from \cref{def_amenable} has constant radius $\frac{1}{a}$ in the relative error metric. Since the set of points where $\kappa=\infty$ includes $X=\{ k\pi \mid 0\neq k\in\Z\}$ and $\dist(k\pi, (k+1)\pi) = \log \frac{k+1}{k}$ is decreasing monotonically to zero as $k\to\infty$, we see that for sufficiently large $x$ any constant-radius ball in the relative metric will contain one or more points with infinite condition number.
Hence, (A.2) fails for $\sin$ on $\R$. This also shows that \cref{lem:domainaditivity} cannot be extended to countable unions.

Since the sine function is not amenable on $\R$, it cannot be realized as any composition of \textit{compatible}, \textit{amenable} functions $g_i$, $i=1,\ldots,\ell$. This is rather unsettling. First, failure of compatibility of amenable functions often results in an unstable algorithm.
To see why, note that the \emph{numerical excess factor}, which for the composition of two functions $g \circ h$ is
\begin{align}\label{eqn_blow_it_up}
 \frac{\tmu(g, h(x)) \tmu(h, x)}{\tmu(g \circ h, x)},
\end{align}
cannot be bounded by a constant polynomial $c$ in \cref{def_compatible}. Recall that a very large numerical excess factor was used in \cite{BBV2019,NT2016} to establish unstability. An numerical excess factor that is not bounded above by a small constant is compatible with the existence of inputs for which the algorithm is forward unstable.
Second, failure of one of the functions $g_i$ to be amenable typically signals a highly sensitive function so that the maximum condition number in a small neighborhood of a point $x$ is greater than a constant times $\kappa(g_i,x)$. This makes it difficult to stably compose them in numerical algorithms. In this light, it is not a surprise that the following Matlab code reveals the dramatic growth of relative forward errors displayed in \cref{fig_sinful}:\\[-6pt]

\begin{quote}\begin{verbatim}
lop = zeros(100,1);
PI = vpa('pi',10000);
SIN1 = sin(vpa('1',10000));
for k = 1 : 100
    trueInput = PI*vpa('2',10000)^(k) + 1;
    numAlgSin1 = sin(double(trueInput));
    lop(k) = abs(log(numAlgSin1 / SIN1)) / eps;
end
\end{verbatim}\end{quote}

We stress that it is impossible to decompose the sine function on $\R$ into a composition of compatible, amenable functions, irrespective of the existence or (in)stability of any potential algorithms realizing these functions.

\begin{figure}[tb]
 \includegraphics[width=\textwidth]{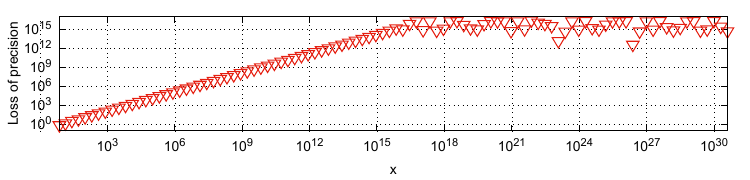}
 \caption{\footnotesize The loss of precision when computing the sine of $x = 1 + \pi 2^k$ for $k=1,\ldots,100$ in Matlab. The reference value was computed using variable precision arithmetic with $10000$ digits.}
 \label{fig_sinful}
\end{figure}

\section{Conclusions} \label{sec_conclusions}
This paper derived a composition theorem for forward stable numerical algorithms, under the sufficient conditions of amenability and compatibility of the functions that these numerical algorithms approximately compute. We employed the introduced tools to efficiently prove the stability of well-known, elementary algorithms.
Finally, we illustrated that not even every smooth function can be decomposed into compatible and amenable functions.

We hope to have convinced the reader that the composition theorem can be a valuable tool in an applied or computational mathematician's toolbox. This theorem gives clear guidance on how problems should be decomposed into subproblems, so that applying forward stable algorithms to the subproblems automatically stably solves the original problem. In our experience, a failure of the composition theorem often signals potential sources of numerical instability that merit further analysis. The composition theorem enables us to fully leverage the extensive and hard-earned catalogue of stable numerical algorithms from the last seven decades in the design of new forward stable algorithms for tomorrow's challenges.

\appendix
\section{Proofs of the technical results}

\subsection{Proof of \cref{lem:mucompupper}}\label{sec:P0}
	Let $h:S\to\R^n$, $g:T\to\R^m$ with $h(S)\subseteq T\subseteq\R^n$. First, assume that $x\in S$ is such that $\kappa(h,x)$ and $\kappa(g,h(x))$ are finite. Then,	for any given $\delta\in(0,1)$, let $\epsilon',\epsilon>0$ be such that
	\[
	\sup_{z\in T,\, 0<\dist(h(x),z)\leq \epsilon}\frac{\dist(g(h(x)),g(z))}{\dist(h(x),z)}\leq \kappa(g,h(x))+\delta, \text{ and}
	\]
	\[
	\sup_{y\in S,\, 0<\dist(x,y)\leq \epsilon'}\frac{\dist(h(x),h(y))}{\dist(x,y)}\leq \kappa(h,x)+\epsilon.
	\]
	Note that $\epsilon',\epsilon>0$ exist from the definition of $\kappa(h,x)$ and $\kappa(g,h(x))$, and note also that without loss of generality we can assume that $\epsilon'<\epsilon<\delta<1$.  Then, for $y\in S$, $0<\dist(x,y)\leq \frac{\epsilon'}{2\tmu(h,x)}$ we have
	\[
	\dist(h(x),h(y))\leq (\kappa(h,x)+\epsilon)\dist(x,y)\leq \frac{\kappa(h,x)+\epsilon}{2\tmu(h,x)} \epsilon' <\epsilon.
	\]
	We thus have
	\begin{multline*}
	\dist(g(h(x)),g(h(y)))
	\leq (\kappa(g,h(x))+\delta)\dist(h(x),h(y))
	\leq\\ (\kappa(g,h(x))+\delta)(\kappa(h,x)+{\epsilon})\dist(x,y).
	\end{multline*}
	We have proved that
	\[
	0<\dist(x,y)\leq \frac{\epsilon'}{2\tmu(h,x)}\Rightarrow\frac{\dist(f(x),f(y))}{\dist(x,y)}\leq (\kappa(g,h(x))+\delta)(\kappa(h,x)+{\epsilon}),
	\]
	which readily implies $\kappa(f,x)\leq (\kappa(g,h(x))+\delta)(\kappa(h,x)+{\epsilon})$. Since the choice of $\delta$ is arbitrary, the claimed bound \cref{eq:boundmu} follows.

Finally, \cref{eq:boundmu} is valid if $\tmu(h,x)=\infty$ or $\tmu(g,h(x)) = \infty$, as $\tmu(g,h(x)) \ge 1$ and $\tmu(h,x)\ge1$. This concludes the proof.\qed

\subsection{Proof of \cref{lem:geodesicas}}\label{sec:P1}
Let $\gamma  : [0,d] \rightarrow \R^n$ be the minimizing geodesic from $x$ to $y$, parameterized by arc-length in such a way that $\dist(\gamma(a),\gamma(b))=b-a$ for $0\leq a\leq b\leq d$, and consider the univariate function $\psi(t) = \dist(f(x),f(\gamma(t))$, which is continuous since $\kappa(f,x)<\infty$ for $x$ in the curve. Given $\delta > 0$, consider the set
\[
S = \{ t : \psi(t) \leq (C+\delta) t   \} \subseteq [0,d]
\]
and let $\sigma = \sup S$; since $0 \in S$ and $S$ is bounded $\sigma$ exists. Moreover $\sigma \in S$ as $S$ is closed. Suppose $\sigma < d$. Then,
\[
\psi(\sigma + \epsilon) \leq \psi(\sigma) + \dist(f(\gamma(\sigma)),f(\gamma(\sigma+\epsilon))).
\]
The first addend is bounded by $(C+\delta)\sigma$. To estimate the second addend observe that, by definition of $C$, the condition number, and $\limsup$, there exists $\eta$ such that whenever $\epsilon < \eta$ then
\[
\dist(f(\gamma(\sigma)),f(\gamma(\sigma+\epsilon)))
\leq \left(C+\frac{\delta}{2}\right) \dist(\gamma(\sigma),\gamma(\sigma+\epsilon))
< (C+\delta) \epsilon.
\]
Hence, for sufficiently small $\epsilon$,
\(\psi(\sigma + \epsilon) < (C+\delta)(\sigma+\epsilon)\)
which implies $\sigma + \epsilon \in S$ and contradicts the definition of $\sigma$. We conclude $\sigma=d$, which proves the lemma.\qed

\subsection{Proof of \cref{prop_cond_expression}}\label{sec:Pm1}
Let
$\R^m = \cup_{j=0}^{3^{m}-1} U_j$ (respectively $\R^n = \cup_{j=0}^{3^{n}-1} V_j$) be the decomposition of the domain $\R^m$ (resp. the codomain $\R^n$) into its connected components induced by the coordinatewise relative error metric. The linear space spanned by the elements of $U_j$ will be denoted by $\overline{U_j}$; it is the subspace of vectors whose $i$th component is zero if $i\not\in\chi(x)$.

Let $j>0$ be fixed and assume $x \in U_j$.  If there is no open neighborhood $N_x \subset U_j$ of $x$ such that $f(N_x) \subset V_k$ for some $k$, then it follows from \cref{def:mu} and \cref{def:finsler} that $\kappa(f,x) = \infty$.
Otherwise,
\begin{align*}
\kappa(f,x)^2
= \lim_{\epsilon \rightarrow 0} \sup_{y\in \overline{U_j}}
\frac{{\underset{i \in \chi(f(x))}{\sum}  \left(\log \frac{|f_i(x + \epsilon y)|}{|f_i(x)|} \right)^2}}{{\underset{i\in\chi(x)}{\sum} \left(\log \frac{|x_i + \epsilon y_i |}{|x_i|} \right)^2 }}
= \lim_{\epsilon \rightarrow 0} \sup_{y \in \overline{U_j}} \frac{{\epsilon^2 \underset{i\in\chi(f(x))}{\sum} \left( \frac{(\deriv{x}{f_i})(y)}{f_i(x)} \right)^2 + o(\epsilon^2)}}{{\epsilon^2 \sum_{i\in\chi(x)}  \left(\frac{y_i}{x_i} \right)^2 + o(\epsilon^2) }},
\end{align*}
having used the definition of the derivative to expand $f_i(x+\epsilon y)$ as $f_i(x)+\epsilon \deriv{x}{f_i} y + o(\epsilon)$.
On the other hand, the square of the right--hand term of \eqref{eqn_mu_diff} equals
\begin{align*}
\kappa(f,x)^2
= \sup_{z \in \overline{U_j}} \frac{{\underset{i\in\chi(f(x))}{\sum} \left( \frac{(\deriv{x}{f_i})(\diag(x)z)}{f_i(x)} \right)^2 }}{{ \sum_{i\in\chi(x)} z_i^2  }}
=\sup_{y \in \overline{U_j}} \frac{{\underset{i\in\chi(f(x))}{\sum} \left( \frac{(\deriv{x}{f_i})(y)}{f_i(x)} \right)^2 }}{{ \sum_{i\in\chi(x)}\left(\frac{y_i}{x_i} \right)^2  }},
\end{align*}
where we have changed coordinates by setting $y=\diag(x)z$. The proposition follows.\qed

\subsection{Proof of \cref{lem:Gronwall}}\label{sec:P2}
Let $x \in S_k$ be arbitrary, $\tmu(f,x)<\infty$, and let $y\in B_x$ be arbitrary. Denote the minimizing geodesic connecting $x$ and $y$ by $\gamma(t) : [0,d] \to \R^{m_k}$, where $d=\dist(x,y) \leq \frac{1}{a(M_k) \tmu(f,x)}$.

First we show that $\gamma \subset V_k$ under the conditions of \cref{lem:Gronwall}. Assume it is not true.
Since $\gamma(0)=x\in V_k\subseteq S_k\setminus\partial S_k$ (using (i)), there exists a supremum $t \in [0,d]$ such that $\gamma(s) \in V_k$ for all $0 \le s < t$ but $\gamma(t) \not\in V_k$.
From item (ii) in the assumptions and the fact that $\tmu(f,\gamma(s)) < \infty$ for all $s\in[0,t)$ because $\gamma([0,t))\subset V_k$, we have that $s \mapsto \tmu(f,\gamma(s))$ is a continuous function.
Let
	\[
	C(s)=\sup_{\alpha\in[0,s)}\tmu(f,\gamma(\alpha)),
	\]
which is a continuous function of $0 \le s < t$. From \cref{lem:geodesicas,eq:alternative} we have, for all $s\in[0,t)$,
	\begin{align*}
	\log \left|\frac{\tmu_f(\gamma(s))}{\tmu_f(x)}\right|
	= \dist( \tmu_f(\gamma(s)), \tmu_f(x) )
	&\leq \sup_{\alpha\in[0,s)}\kappa(\tmu_f,\gamma(\alpha)) s \\
	&\leq \frac{a(M_k)}{4}\sup_{\alpha\in[0,s)}\tmu(f,\gamma(\alpha)) s
	\leq \frac{a(M_k)}{4} C(s) s.
	\end{align*}
	In particular, since $0 < t \le d \le \frac{1}{a(M_k)\tmu(f,x)}$, it follows that
\[
\log \left|\frac{C(s)}{\tmu_f(x)}\right| \leq \frac{a(M_k)}{4} C(s) s\leq \frac{C(s)}{4\tmu(f,x)},\quad s\in[0,t).
\]
This means that the set $R=\{C(s)/\tmu(f,x):s\in[0,t)\}$ is contained in the set $\{\alpha\in (0,\infty):\log(\alpha)\leq \alpha/4\}$, which is itself contained in $(0,2]\cup [8,\infty)$. Since $R$ is connected (as $C$ is continuous) and contains the point $1$ (since $C(0)=\tmu(f,x)$), we must have $R\subseteq(0,2]$. That is, $C(s)\leq 2\tmu(f,x)$ for all $s\in[0,t)$. Since the hypothesis includes that the condition number explodes to $\infty$ when $\gamma(s)$ approaches the boundary $\partial S_k$ or the limit point has infinite condition number, this contradicts the existence of $t \in [0,d]$ such that $\gamma(t) \not\in V_k$. We conclude that $\gamma \subset V_k$, and hence $B_x$, the geodesic ball of radius $\frac{1}{a(M_k) \tmu(f,x)}$, is contained in $V_k\subseteq S_k$ as well. This establishes the first item in \cref{def_amenable}.

As the argument above applies for all $t\in[0,d]$ we can also conclude
\[
\tmu(f,y)\leq C(d)\leq 2 \tmu(f,x),
\]
which establishes the second item in \cref{def_amenable}. This concludes the proof.\qed

\section{The BSS model of computation}\label{appendix:bss}
We recall in this appendix some known definitions to answer formal questions like ``what is an algorithm?'' and ``what is an approximate computation?''

\subsection{The Blum--Shub--Smale model of computation}

The BSS computational model \cite{BSS1989,BCSS1998} is a formalization of the concept of {\em algorithm}, similar to the classic Turing machine but permitting exact computation between real numbers. Indeed, a BSS machine is sometimes called a real number Turing machine. Formally, a BSS machine $M(\bar{I},\bar{O},\bar{S},G)$ in canonical form consists of an input space $\bar{I}$, an output space $\bar{O}$, a state space $\bar{S}$, and a connected, finite directed graph $G=(V,E)$ with nodes $V$ and edges $E$. The state space can be thought of as an ``internal memory'' that can grow as large as wanted: an element $\bar{s}\in\bar{S}$ is a bi-infinite sequence of real numbers $(\ldots,s_{-2},s_{-1},s_0,s_1,s_2,\ldots)$ where all but a finite quantity of entries are zero. Following \cite{MalajovichShub}, a node $v\in V$ of $G$ is of one of the following types:
\begin{enumerate}
\item \textit{Input}: this is the unique node where the input of the machine is written. It is characterized by having no incoming edges. It takes the input $\bar{i}\in\bar{I}$ and sends it to some state $\bar{s}\in \bar{S}$.
\item \textit{Output}: characterized by having no outgoing edges. Once the output node has been reached, the machine terminates and the current state $\bar{s}\in \bar{S}$ is converted to an output $\bar{o}\in \bar{O}$.
\item \textit{Computation}: these have exactly one outgoing edge, and can perform an operation on the current state $\bar s\in \bar S$, changing it to some other $\hat s\in \bar S$. An ``operation'' is one of the basic operations $\{+,-,\times,/\}$ executed on two elements of $\bar s$, or copying one number in another place of the memory, or adding a built-in constant of the machine. All these operations produce exact answers.
\item \textit{Branch}: these have two outgoing edges. If $s_0>0$ then the first outgoing edge is taken, otherwise the second outgoing edge is taken.
\item ``\textit{Fifth node}:'' these move the state space elements. All elements of $\bar s=(\dots,s_{-2},s_{-1},s_0,s_1,s_2,\ldots)$ are shifted either one step right or left.
\end{enumerate}
We label the nodes as follows. Let $N$ be the number of non-terminating nodes in the computation graph $G$. The input node is $v_0 \in V$. The $N'$ output nodes are $v_{N+1}, v_{N+2}, \ldots, v_{N+N'}$. All other nodes are numbered from $1$ to $N\in\N$.
The sequence of states $(0,v_0,s(0)),(1,v_{i_1},s(1)),\ldots$ is called the \textit{exact computation sequence} with input $\bar{i}$ if $s(0)=\bar i$ is the input and each $(j,v_{i_j},s(j)) \in \N \times V \times \bar{S}$ such that $(v_0,v_{i_1})\in E$, $(v_{i_j},v_{i_{j+1}})\in E$ for all $j$, and the correct branch is taken in branch nodes.
The BSS machine terminates for input $\bar{i}$ if its exact computation sequence is finite.
At a branch node $j$, the two possible outgoing edges can be thought of as ``go to line $k$'' instructions, thus allowing loops and conditionals. The purpose of the fifth nodes is to access to arbitrary positions in the sequence. Since we think of $\bar{S}$ as an internal random access memory and of elements $\bar{s}$ as finite collections of variables with assigned values, fifth nodes allow to recall any such variable from memory.

A BSS machine can be described by the computation graph $G$ which describes the program flow in a graphical way or by any reasonable pseudocode, with the caveat that \emph{exact computation of rational functions with real numbers is allowed.} The use of other functions such as for example $\sqrt{x}$, $\log x$, and $\sin x$ is not permitted in a pseudocode description of a BSS machine; these methods should be implemented by separate subroutines.

We cannot be sure that a given machine will finish on every input, nor can we guarantee that the running time will be bounded by some constant independent of the input. A good part of the theoretical efforts in the study of BSS machines has been devoted to understand deep questions related to these points, see \cite{BCSS1998,BC2013}. These subtle issues are not featured in this paper, but we must be aware that a machine is not a function defined on $\bar I$, since it may produce no output (i.e., it may run forever) in some input. The set of points $\Omega\subseteq\bar I$ where a machine produces some output is called the {\em halting set} of the machine, and thus we can capture a part of the machine by considering it as a function defined on $\Omega$, but $\Omega$ is unknown a priori.

Note also that the result of a computation node \emph{might not be well defined for some values of the input} since division by $0$ may occur. Thus, formally, before any computation node where a division is performed we include a branch node that checks whether the denominator equals $0$; if so, the branch node repeats forever so that the machine will never finish on that input.

\subsection{The standard model of floating-point arithmetic}\label{sec_flp_model}

We recall the usual properties assumed of floating-point arithmetic \cite{Higham1996}. A floating-point number system $\F\subseteq\R$ with base $2$ is a
subset of the real numbers of the form
\[
\pm m\cdot 2^{e-t},
\]
where $t\in\Z$, $t>2$,
 is the \textit{precision} and $e\in\Z$, $\emin\leq e\leq \emax$ is bounded. Moreover, the mantissa $m\in\Z$ is either $0$ or satisfies $2^{t-1}\leq m\leq 2^t-1$, ensuring a unique representation. In other words, $x\in\F$ if $x=0$ or $x$ is of the form
\(
x=\pm 2^e\cdot [0.a_1\ldots a_t]_2
\)
for some $e$ in the range and $a_1,\ldots,a_t\in\{0,1\}$, $a_1\neq0$.
 The \textit{unit roundoff} is $u=2^{-t}$.

With these definitions and the assumption about $\emin=-\infty$ and $\emax=\infty$, the floating-point number system $\F_u$ satisfies the well-known axioms:
\begin{itemize}
    \item For all $x\in\R$, $\fl(x) \in \F_u$.
	\item For all $x\in\R$, \(\fl(x)=x(1+\delta)\) for some real $|\delta|\leq u$.
	\item For all $x\in\F_u$, $\fl(x) = x$.
	\item For all $x\in\F_u$, $\fl(-x)=-\fl(x)$.
	\item If $u'<u$ and $x=\fl(x)$, then $x=\mathrm{fl}_{u'}(x)$.
	\item If $x,y\in\R$ and $x\leq y$, then $\fl(x)\leq \fl(y)$.
	\item For every operation $\circ \in \{+_\R,-_\R,\cdot_\R,\div_\R\}$, there is a corresponding floating-point operation $\hat{\circ} : \F_u \times \F_u \to \F_u$ such that for all $x, y \in \F_u$ we have that
	\[
	x \operatorname{\hat{\circ}}y = (x \circ y)(1+\delta) \text{ for some real } |\delta|\leq u,
	\]
	with the unique exception of division by $0$ that produces either $\mathrm{NaN}$ or $\pm\infty$.
\end{itemize}

\subsection{Approximate computations}
Numerical analysis studies the behaviour of algorithms in the presence of roundoff errors introduced by floating-point arithmetic. In recent works \cite{cucker2015,MalajovichShub}, the BSS model was used to formalize intuitive concepts commonly used in the study of algorithms with a focus on complexity theory for approximate computations. Both define a \textit{numerical algorithm} as a BSS machine whose computations are made in floating-point arithmetic.\footnote{Cucker \cite[Definition 2]{cucker2015} gives a definition that is similar to that of \textit{weak and strong approximate computations} in \cite[Section 7]{MalajovichShub}, although there are some subtle differences in their definitions of complexity classes. Our theory is quite robust and does not critically depend on these details.} We recall the definition based on strong approximate computations from \cite{MalajovichShub}.

\begin{definition}[\protect{Strong approximate computation \cite[Definition 7.1]{MalajovichShub}}]\label{def:FPalgorithm}
Consider the BSS machine $M(\bar{I},\bar{O},\bar{S},G)$ and let $0<u<\frac{1}{4}$ be a unit roundoff. The \emph{strong $u$-computation} for $M$ on input $\bar{i}\in \bar{I}$ is the exact computation sequence $(j,v_{i_j},\widehat{s}(j))$ obtained by the BSS machine $M^u(\bar{I},\bar{O},\bar{S},\widehat{G})$, where $\widehat{G}$ is constructed from $G$ by replacing all
\begin{itemize}
 \item real operations $x \circ y$ with $\circ \in \{+_\R,-_\R,\cdot_\R,\div_\R\}$ by their floating-point counterpart $x \operatorname{\hat{\circ}} y$,
 \item real constants $c\in\R$ by $\fl(c)$, and
 \item real inputs or outputs $x\in\R$ by $\fl(x)$.
\end{itemize}

For every precision $t>2$ (or, equivalently, roundoff $u<\frac{1}{4}$) the machine $M^u$ has access to $u$ and $t$ as internal constants, which formally means that its input is $(\bar i,u,t)$. The halting set $\Omega^u$ is the set of inputs for which the strong $u$-computation terminates. If $\bar i\in \Omega^u$ we write $M^u(\bar i)$ for the output (i.e., the last state) of the strong $u$-computation. With this notation, we can look at $M^u$ as a function $M^u:\Omega^u\to\bar O$. We call $M^u$ a \emph{numerical algorithm}.
\end{definition}

Two different BSS machines with the same output will in general produce distinct strong $u$-computation sequences with possibly other outputs for the same $x$.
Even the halting sets of two such computations can be different.
This is because the roundoff errors made in each floating-point operation can accumulate in various ways; a classic example is the machine computing $x\to x+1$ by adding $1$ and another doing the same by adding $0.1$ ten times.

As pointed out in \cite{cucker2015}, since integer numbers can be represented exactly in a $b$-ary expansion, we can assume that a strong $u$-computation differs from the original exact computation of the BSS algorithm only in its treatment of reals and the real operations. Integers and the elementary integer operations $\circ_\Z \in \{ +_\Z, -_\Z, \cdot_\Z, \div_\Z \}$ need \emph{not} be replaced by floating-point operations.

\bibliographystyle{siam}
\bibliography{strings,stability}

\end{document}